\documentclass[11pt,onesided]{amsart}   	
\usepackage{geometry}                		
\usepackage{graphicx}				
\usepackage{amssymb,mathrsfs}
\usepackage{amssymb,amsthm,amsmath,stmaryrd}
\usepackage{tikz}
\usepackage{tikz-cd}
\usepackage{color}
\usepackage{eucal,accents,upgreek,enumerate}
\usepackage[headings]{fullpage}
\usepackage{bm}

\def\trop{\mathrm{trop}}
\newcommand{\spec}{\operatorname{Spec}}

\newcommand{\CC} {{\mathbb C}}          
            
\newcommand{\NN} {{\mathbb N}}		
\newcommand{\PP}{\mathbb{P}}         
		
\newcommand{\RR} {{\mathbb R}}		
\newcommand{\ZZ} {{\mathbb Z}}		
\newcommand{\defi} {\mathrm{def}}

\newtheorem{theorem}{Theorem}[section]
\newtheorem*{theorem*}{Theorem}
\newtheorem{proposition}[theorem]{Proposition}
\newtheorem{corollary}[theorem]{Corollary}
\newtheorem{lemma}[theorem]{Lemma}
\newtheorem{definition}[theorem]{Definition}

\newtheorem*{lemma*}{Lemma}
\newtheorem*{remark*}{Remark}
\newtheorem*{example*}{Example}

\newtheorem{innercustomthm}{Theorem}
\newenvironment{customthm}[1]
  {\renewcommand\theinnercustomthm{#1}\innercustomthm}
  {\endinnercustomthm}

\usepackage[backref]{hyperref}
\hypersetup{
  colorlinks   = true,          %Colors links instead of ugly boxes
  urlcolor     = blue,          %Color for external hyperlinks
  linkcolor    = purple,          %Color of internal links
  citecolor   = blue             %Color of citations
}

\tikzset{main node/.style={circle,fill=blue!20,draw,minimum size=1.3cm,inner sep=0pt},
            }          
\tikzset{marked point/.style={circle,fill=red},
            }
\usetikzlibrary{calc}

\address{{\bf  Yoav Len}\newline Department of Mathematics, University of Waterloo, Waterloo, ON, Canada}
\email{yoav.len@uwaterloo.ca}

\address{{\bf Dhruv Ranganathan}\newline Department of Mathematics, Massachusetts Institute of Technology, Cambridge, MA, USA}
\email{dhruvr@mit.edu}

\title[Elliptic curves on toric surfaces]{{\larger \larger  E}numerative geometry of elliptic curves on toric surfaces}

\author[Dhruv \& Yoav]{{\larger Y}{\smaller oav} {\larger L}{\smaller en} \& {\larger D}{\smaller hruv} {\larger R}{\smaller anganathan}}
\begin{document}

\maketitle

%\vspace{-1in}
\begin{abstract}
We establish the equality of classical and tropical curve counts for elliptic curves on toric surfaces with fixed $j$-invariant, refining results of Mikhalkin and Nishinou--Siebert. As an application, we determine a formula for such counts on $\mathbb P^2$ and all Hirzebruch surfaces. This formula relates the count of elliptic curves with the number of rational curves on the surface satisfying a small number of tangency conditions with the toric boundary. Furthermore, the combinatorial tropical multiplicities of Kerber and Markwig for counts in $\mathbb P^2$ are derived and explained algebro-geometrically, using Berkovich geometry and logarithmic Gromov--Witten theory. As a consequence, a new proof of Pandharipande's formula for counts of elliptic curves in $\mathbb P^2$ with fixed $j$-invariant is obtained.
\end{abstract}

\section{Introduction}
Let $X = X(\Sigma)$ be a proper toric surface with fan $\Sigma$ and let $\Delta$ be a convex polytope giving an equivariant polarization for $X$. Let $N(\Delta)$ denote the number of elliptic curves $E$ in $X$ having fixed generic $j$-invariant, lying in the linear system $\PP_\Delta$, and passing through the expected number of points in general position. Let $N^\text{trop}(\Delta)$ be the number of tropical genus $1$ curves on the surface $X^\trop$ with toric degree $\Delta$, fixed cycle length, and passing through the expected number of points in general position, see Definition~\ref{weights}. Our first result is a new correspondence theorem in this setting, extending a result of Mikhalkin~\cite{Mikhalkin}.

\begin{customthm}{A}\label{thmA}
The number $N^\text{trop}(\Delta)$ does not depend on the cycle length or the point configuration. Furthermore, for $j \neq 0,1728$ we have $N^\text{trop}(\Delta) = N(\Delta)$.
\end{customthm}

\noindent
We apply this correspondence theorem in the setting of Hirzebruch surfaces $\mathbb{F}_n$ to count elliptic curves with a fixed $j$-invariant. 

\begin{customthm}{B}\label{thmB}
There exists an explicit formula for the number of elliptic curves with fixed $j$-invariant of degree $(a,b)$ on $\mathbb{F}_n$ through $2b+(n+2)a - 1$ points in general position in terms of genus $0$ counts on $\mathbb F_n$ with $1$ or $2$ toric tangency conditions, together with combinatorial coefficients. 
\end{customthm}

\noindent See Theorem \ref{mainTheorem} for a precise statement of the formula. Our work is inspired by an elegant result of Pandharipande, who computes the number of elliptic curves in the plane with fixed $j$-invariant from Kontsevich's formula for rational curves in the plane~\cite{pand}. For Hirzebruch surfaces, the analogous role is played by counts of rational curves with $1$ or $2$ prescribed tangency conditions with the toric boundary.

\subsection{Further discussion} We approach Theorem~\ref{thmA} from the Abramovich--Chen--Gross--Siebert theory of logarithmic stable maps to toric varieties, which in turn builds on the framework of Nishinou and Siebert~\cite{AC11,Che10,GS13,NS06}. Our proof follows a similar line of reasoning as the re-proof of their genus $0$ classical/tropical correspondence theorem due to the second author~\cite{R15b}. The central idea is that tropical maps encode the combinatorial data in logarithmic special fibers of degenerating families of algebraic maps. When the degeneration is maximal, we explicitly determine the number of logarithmic lifts of this data, and the multiplicities resulting from the associated smoothing problem can be computed combinatorially, using analytic domains in a Berkovich space. However, there is an important subtlety that arises in the present context. For correspondence theorems in other settings, such as those considered in~\cite{CMR14a,Mikhalkin,NS06,Tyo12}, general position arguments can be used to force the relevant tropical curves to be realizable by algebraic ones. When fixing the $j$-invariant however, this is no longer true, and we must appeal to Speyer's \textit{well-spacedness} condition~\cite[Theorem 3.4]{Spe14} to understand which tropical curves can be lifted. The result thus presents the first nontrivial calculation of enumerative invariants in the presence of superabundant geometries, and outlines a conceptual framework that we expect to be useful in future applications.

The results of this paper build upon work of Kerber and Markwig~\cite{KM}, who study tropical elliptic curves of fixed $j$-invariant in $\PP^2_{\text{trop}}$. Without the aid of a correspondence theorem, they observed after the fact that their formulae agreed with those of Pandharipande~\cite{pand}. Remarkably, although the counts agree there does not exist a direct correspondence theorem between classical curves on $\PP^2$ and Kerber and Markwig's tropical curves on $\PP^2_{\trop}$. It arises from the analysis in this text that they count superabundant curves with positive multiplicity that do not satisfy Speyer's well-spacedness condition, and hence are non-realizable. Conversely, they assign multiplicity zero to several curves that do satisfy Speyer's condition. Nonetheless, one can first show, combinatorially, that Kerber and Markwig's tropical counts are equal to those considered in this paper, and then apply Theorem~\ref{thmA} to obtain a new and ``purely tropical'' proof of Pandharipande's formula. 

There are a number of related results concerning the enumerative geometry of Hirzebruch surfaces. Counts of rational curves on $\mathbb F_2$ were considered by Abramovich and Bertram~\cite{AB} and were related to counts on $\PP^1\times \PP^1$. This result was extended to arbitrary genus by Vakil~\cite{V00}. In~\cite{FM} Franz and Markwig produce a tropical proof of Vakil's formula. Tropical techniques were then used by Brugall\'e and Markwig~\cite{BM} to give a general formula relating the enumerative invariants of $\mathbb F_n$ and $\mathbb F_{n+2}$. Recently, floor diagram calculus on $\mathbb F_n^{\trop}$ has been used to study polynomiality properties of ``double'' Gromov--Witten invariants of Hirzebruch surfaces~\cite{AB14}, and to relate the (refined) Severi degrees of Hirzebruch surfaces to matrix elements in Fock spaces~\cite{BG14}. The connections to Fock spaces were first recognized by Cooper and Pandharipande~\cite{CP12}, and it would be natural to consider the role of the Fock space formalism in the present setting.

The enumerative geometry of curves with fixed generic complex structure has also seen substantial interest, and generalizations of Pandharipande's results are known in several other settings. These generalizations are essentially orthogonal to the results of the present text, but present natural avenues for future inquiry in tropical and logarithmic geometry. In the late 1990's, Ionel used analytic and symplectic methods to deduce formulae for the numbers of elliptic curves in projective space of any dimension~\cite{Ionel98}. In the early 2000's, Zinger gave beautiful formulae for the number of genus $2$ curves with fixed complex structure in $\PP^2$ and $\PP^3$, and for genus $3$ curves with fixed complex structure in $\PP^2$, see~\cite{Zing04,Zing05} and also~\cite{KQR}. Finally, during the preparation of this manuscript, we learned of work of Biswas, Mukherjee, and Thakre~\cite{BMT} who study the enumerative geometry of elliptic curves on del Pezzo surfaces with fixed $j$-invariant. Their method is to first compute a Gromov--Witten invariant, and then find an exact solution for the contracted component contribution as an intersection number on $\overline{\mathcal M}_{0,n}$. We expect that both methods will be useful in the future, possibly in unison. The overlapping case of $\mathbb F_1$ yields the same formula, see Corollary~\ref{cor: f1-formula}. We expect that further refinements of the methods here will lead to algebraic and tropical approaches to many of these results, as well as generalizations to target toric varieties.

%\begin{remark}\textnormal{
%In~\cite{KM}, Kerber and Markwig ask why there is no correspondence theorem for curves whose $j$-invariant is $0$ or $1728$. We record a simple geometric explanation. Given a realizable tropical elliptic curve of genus $1$, one can associate a family of elliptic curves $\mathscr E$ over $\spec(\CC\llbracket t\rrbracket)$, together with a map to $X$. Since the tropical curve has genus $1$, the loop must have a positive length $\ell$. The generic fiber $\mathscr E_\eta$ is an elliptic curve over $\CC(\!(t)\!)$ with multiplicative reduction. Moreover, the $j$-invariant must have negative valuation equal to $\ell$. In particular, $j(E_\eta)\neq 0,1728$, so its automorphism group is isomorphic to $\ZZ/2$.}
%\end{remark}

\subsection*{Acknowledgements} We are especially grateful to our advisor Sam Payne for creating an environment for us to learn and work together, and to Hannah Markwig for many useful discussions in the early stages of the project. We are grateful to Mark Gross, Hannah Markwig, and Rahul Pandharipande for remarks on a previous draft. Thanks are also due to Dan Abramovich, Dori Bejleri, Ritwik Mukherjee, and Ilya Tyomkin. This project was initiated when the second author was visiting Brown University, and important progress was made while we were in attendance at the 2015 AMS Algebraic Geometry Summer Institute. Final revisions were made while the second author was visiting the Institute for Advanced Study in 2017. It is a pleasure to acknowledge all these institutions for stimulating working conditions.  We thank the referee for several helpful comments that led to a number of improvements in the text.\

\noindent
Y.L. was supported by DFG grants MA 4797/1-2 and MA 4797/3-2. D.R. was supported by NSF grant CAREER DMS-1149054 (PI: Sam Payne).

\section{The moduli space of tropical stable maps}\label{sec: moduli-maps}

Let $M$ be a lattice of rank $2$ and $T = \spec(\CC[M])$ the associated torus. Let $M^\vee$ be the dual lattice. Fix a lattice polygon $\Delta$ in $M_\RR$ and let $\Sigma$ denote its normal fan, defining a toric variety $X = X(\Sigma)$ with dense torus $T$. 

We assume familiarity with the notions of abstract and parametrized tropical curves (namely, \textit{tropical stable maps}) in $M^\vee_\RR$. See~\cite[Section 2]{KM} or~\cite[Section 2]{R16} for the relevant definitions. We remind the reader that for a parametrized tropical curve $[f: \Gamma\to \Sigma]$, the \textit{direction of an edge $e$} is the direction vector of the affine line onto which $e$ maps. The slope $w$ of $f$ upon restriction to $e$ is called the \textit{weight of the edge $e$}. The unbounded edges of a tropical curve will be referred to as \textit{ends}. By a tropical stable map of \emph{degree} $\Delta$ we mean a parametrized tropical curve in $\Sigma$ whose set of ends is dual to $\Delta$. Any such curve determines an extended metric graph mapping to the extended tropicalization $\overline \Sigma$ of $X$, whose infinite edges are transverse to the boundary $\overline \Sigma\setminus \Sigma$.

\subsection{The tropical space of maps} Let $D$ be the number of ends of a curve of degree $\Delta$, and define $N=D-1$. Let $\mathcal  M_{1,N}'(\Delta)$ be the set of isomorphism classes of parametrized tropical curves of genus $1$ of type $\Delta$ with $N$ marked points. This space has the structure of a cone complex, obtained by gluing cones $\mathcal  M_{1,N}(\Delta)^\alpha$ corresponding to combinatorial types $\alpha$. 

Let $\alpha$ be a combinatorial type for a map represented by $[f:\Gamma\to \Sigma]$ such that the cycle of $\Gamma$ is mapped to a line in $|\Sigma|$. Subdivide the cone $\mathcal  M_{1,N}(\Delta)^\alpha$ along the locus where the map is \textit{well-spaced}, i.e. the two edges emanating from the image of the cycle have equal length, see Figure~\ref{flatLoopSubdiv}. After this subdivision, define $\mathcal  M_{1,N}(\Delta)$ be the union of cells of dimension at most $2D-1$. This gives the set $\mathcal  M_{1,N}(\Delta)$ the structure of a pure dimensional cone complex. See~\cite[Section 3]{KM} for further details and~\cite[Section 3]{CMR14a} for the analogous construction for target curves. Generalities on colimits of cone complexes may be found in~\cite[Section 2]{ACP}. 

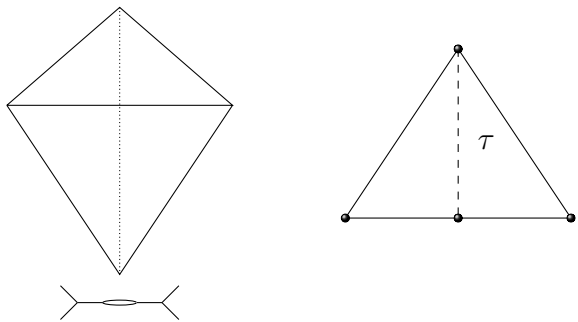
\begin{figure}[h!]
\begin{tikzpicture}[font=\tiny,scale=.75]
\begin{scope}[yshift = -3.5cm, xshift = -0.75cm, scale=0.3] %The tropical curve
\draw (0,0)--(1.5,0);
\draw (1.5,0) arc (179:10:1 and 0.15);
\draw (1.5,0) arc (-179:-10:1 and 0.15);
\draw (3.5,0)--(5,0);

\draw (-1,1)--(0,0);
\draw (-1,-1)--(0,0);
\draw (5,0)--(6,1);
\draw (5,0)--(6,-1);

\end{scope}

\draw (-2,0)--(2,0)--(0,1.737,0)--(-2,0);
\draw (-2,0)--(0,-3)--(2,0);
\draw [densely dotted] (0,1.737)--(0,-3);

\draw (4,-2)--(8,-2)--(6,1)--(4,-2);
\draw [dashed] (6,1)--(6,-2);

\draw [ball color=black] (4,-2) circle (0.7mm);
\draw [ball color=black] (8,-2) circle (0.7mm);
\draw [ball color=black] (6,1) circle (0.7mm);
\draw [ball color=black] (6,-2) circle (0.7mm);

\node at (6.5,-0.65) {\large $\tau$};

\end{tikzpicture}
\caption{On the left is depicted the $3$ dimensional cone corresponding to the flat cycle. On the right, the $2$-dimensional slice obtained by normalizing the total edge length to be $1$. The dashed line and the associated cone $\tau$ is the locus of well-spaced curves.}
\label{flatLoopSubdiv}
\end{figure}

\subsection{Weights on $\mathcal  M_{1,N}(\Delta)$} We now give $\mathcal  M_{1,N}(\Delta)$ the structure of a \textit{weighted} cone complex, by assigning positive rational weights to its maximal cells. The weights account for the number of algebraic curves tropicalizing to a fixed algebraic curve. 

Given a tropical stable map $[f: \Gamma\to \Sigma]$, we define the \textit{deficiency} of $f$, denoted $\mathrm{def}([f])\in \{0,1,2\}$ to be the dimension of the smallest affine subspace of $M^\vee_\RR$ containing the image of the cycle in $\Gamma$. The deficiency is constant within a combinatorial type, so we define $\mathrm{def}(\alpha)$ in the obvious way.

\begin{definition}\label{weights}
Let $\alpha$ be a combinatorial type such that  $C=\mathcal  M_{1,N}(\Delta)^\alpha$  is a maximal cell of $\mathcal  M_{1,N}(\Delta)$, that is, $\dim(C)=2N+1$.
We associate a weight to the cell $C$ according to the deficiency of $\alpha$:
\begin{itemize}
\item {\sc Deficiency $0$}. The cell $C$ is full dimensional exactly when $\alpha$ is a trivalent type. A subset of the coordinates on $C$ are given by the lengths of the source of $\Gamma$. The lengths for edges comprising the cycle cannot be arbitrary, as the cycle must close. This condition is given by a linear map
\[
A =  \binom{a_1}{a_2}:\mathbb{Z}^{2+E(G)}\to\mathbb{Z}^2.
\]
Define the weight of $C$ to be the index of the image of $A$ in $\ZZ^2$. 

\item  {\sc Deficiency $1$}. In this case, $\alpha$ is trivalent and has a well-spaced flat cycle, see Figure~\ref{flatCycleWeight}. By balancing, if another edge emanates from the cycle, then it must be a marked point. 
If the weights $w',w''$ on the upper and lower arc of the cycle are different, or there is a marked point on the cycle, assign the curve weight $\mathrm{gcd}(w',w'')$. Otherwise assign weight $\frac{1}{2}\cdot\mathrm{gcd}(w',w'')=\frac{1}{2}w'=\frac{1}{2}w''$. 

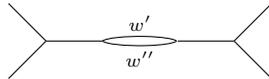
\begin{figure}[h!]
\begin{tikzpicture}[font=\tiny,scale=.5]

\begin{scope} %The tropical curve
\draw (0,0)--(1.5,0);
\draw (1.5,0) arc (170:10:1 and 0.15);
\draw (1.5,0) arc (-170:-10:1 and 0.15);
\draw (3.5,0)--(5,0);

\node [above] at (2.5,0) {$w'$};
\node [below] at (2.5,0) {$w''$};

\draw (-1,1)--(0,0);
\draw (-1,-1)--(0,0);
\draw (5,0)--(6,1);
\draw (5,0)--(6,-1);

\end{scope}

\end{tikzpicture}
\caption{A tropical stable map with a flat cycle.}
\label{flatCycleWeight}
\end{figure}

\item {\sc Deficiency $2$}. If the contracted loop is part of a $5$-valent vertex $v$, the image of $v$ is dual to a triangle in the Newton subdivision. Assign this curve weight equal to the number of interior lattice points of  the triangle. 
If the contracted loop is part of  a 4 valent vertex, then the two other edges emanating from the vertex are parallel due to the balancing condition. This line is dual to an edge of the Newton polygon, and assign the type weight $\frac{L-1}{2}$, where $L$ is the lattice length of the edge. In all other cases assign weight zero.
\end{itemize}

\begin{figure}[h!]
\begin{tikzpicture}[font=\tiny,scale=.5]

\begin{scope} %The tropical curve
\draw (0,0)--(5,0);

\draw (-1,1)--(0,0);
\draw (-1,-1)--(0,0);
\draw (5,0)--(6,1);
\draw (5,0)--(6,-1);

\draw [ball color=black] (2.5,0) circle (0.5mm);

\path (2.5,0) edge[ out=140, in=40
                , looseness=0.6, loop, dashed,
                 distance=3cm]
            (2.5,0);

\end{scope}

\begin{scope}[xshift=8.5cm] 
\draw (0,0)--(5,0);

\draw (-1,1)--(0,0);
\draw (-1,-1)--(0,0);
\draw (5,0)--(6,1);
\draw (5,0)--(6,-1);

\draw [ball color=black] (0,0) circle (0.5mm);

\path (0,0) edge[ out=120, in=40
                , looseness=0.6, loop, dashed,
                 distance=3cm]
            (0,0);

\end{scope}

\end{tikzpicture}
\caption{Two curves with contracted loop edges attached to (left) $4$-valent and (right) $5$-valent vertices.}
\label{flatLoopWeight}
\end{figure}
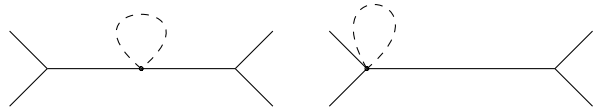

\end{definition}

Given a marked end $p_i$, for $i\in\{1,\ldots, N\}$, there is an evaluation morphism sending $[\Gamma\to \Sigma]$ to a point of $\Sigma$. In addition, there is a ``complex'' structure morphism $j: \mathcal M_{1,N}(\Delta)\to \mathcal M_{1,1}^{\trop}$, which for a given parametrized tropical curve $[\Gamma\to \Sigma]$,  forgets the map, and all but one of the marked point. The choice of marked point plays no role. 
We combine these to a single morphism of cone complexes
\[
J = \prod_i \text{ev}_i\times j:\mathcal M_{1,N}(\Delta)\to \Sigma^N \times \mathcal M^{\trop}_{1,1}.
\]
\noindent
Note that, as a set, the target may be considered as $\mathbb{R}^{2N}\times\mathbb{R}_{\geq 0}$, and this will suffices for our purposes. Standard arguments in tropical geometry, for instance~\cite{KM,NS06}, show that $J$ generically has finite fibers. The \textit{multiplicity} of a tropical curve is defined to be the product of the weight of the associated cell with the determinant of the linear map obtained by restricting $J$ to the cell.

When a tropical stable map has positive deficiency, its image naturally determines a \textit{genus $0$} tropical curve in $\Sigma$. The balancing condition uniquely determines the weights on the edges. In such cases, the multiplicity for $J$ can be expressed in terms of that of the genus $0$ curve as we now describe.

\begin{lemma}\label{multDet}
{\sc (Deficiency $0$)} Suppose $\alpha$ is a deficiency $0$, trivalent type. Then the multiplicity of $\alpha$ is the determinant of the map 
\[
\prod_{i=1}^N \mathrm{ev}_i \times j\times a_1\times a_2:\mathbb{R}^{2+2\cdot E_0}\to\mathbb{R}^n\times\mathbb{R}_{\geq 0}\times\mathbb{R}\times\mathbb{R}.
\]
\end{lemma}
\begin{proof}
See~\cite[Remark 4.8]{KM}.
\end{proof}

\begin{lemma}\label{multFlat}
{\sc (Deficiency $1$)} Let $\Gamma$ be a deficiency $1$ tropical stable map with non-zero weight, such that there is no marked point on the flat cycle. Let $w',w''$ be the weights of the edges forming the cycle, and let $\Gamma'$ be the resulting rational curve. If $w'\neq w''$ then $\text{Mult}(\Gamma) = 2\cdot(w'+w'')\cdot\text{Mult}(\Gamma')$, and otherwise $\text{Mult}(\Gamma) = (w'+w'')\cdot\text{Mult}(\Gamma')$.
\end{lemma}

\begin{proof}
Denote 
$$
\epsilon = \left\{
     \begin{array}{lr}
       1 &  w'\neq w'' \\
       \frac{1}{2} &  w'= w''
     \end{array}
   \right.
   .
$$

\noindent Let $C \subseteq\mathbb{R}^{2N+3}$ be the cell of  $\Gamma$. To compute the determinant of $J$, we need to choose a lattice basis for $C$. Choose a basis for $\mathbb{R}^{2N+3}$ consisting of a unit vector $u_e$ for each bounded edge $e$, and vectors $u_x,u_y$ corresponding to a choice of root vertex $(x,y)$ for $\Gamma$ in the plane. 
Let $e',e''$ be the direction vectors of the arcs forming the cycle. Write $e'=m'\cdot\tilde{e}$  and $e''=m''\cdot\tilde{e}$ with $\text{gcd}(m',m'')=1$ for some vector  $\tilde{e}$. Let $e_0=e'+e''$ be the direction vector of the edges on either sides of the cycle. Finally, let $e_1,\ldots,e_k$ be the  direction vectors corresponding to the rest of bounded edges, where $k=2N-3$. A lattice basis for the cell is given by $u_x,u_y,u_{e_1},\ldots, u_{e_k}, u_{e_0}$ and $e_c = m'\cdot u_{e''}+ m'\cdot u_{e'}$. Let $A$ be the matrix representing $J$ in this basis. The columns of $A$ correspond to each of the basis vectors above, $A$ has a row corresponding to the $j$-invariant and a row corresponding to each of the marked points. The row corresponding to the $j$-invariant consists of zeroes except for the  $e_c$-column where it is $m'+m''$. It follows that $\det(A)$ is $(m'+m'')\cdot\det(A')$, where $A'$ is the matrix obtained by removing the $j$-invariant column and the $e_c$-column. Since there is no marked point on the cycle, the path to any point beyond it passes through both the edges adjacent to the cycle. It follows that by dividing the column corresponding to these edges by $2$, we obtain a matrix $A''$ describing the evaluation map for the curve obtained by replacing the cycle and adjacent edges by a single edge, namely, $\Gamma'$. By \cite[Lemma 3.8]{WDVV}, $\det(A'')$ equals the multiplicity of $\Gamma'$, and we obtain  
$$
\text{Mult}(\Gamma)=\epsilon\cdot\text{gcd}(w',w'')\cdot\det(A)=2\epsilon\cdot\text{gcd}(w',w'')\cdot(m'+m'')\cdot\det(A'')=2\epsilon\cdot(w'+w'')\cdot\text{Mult}(\Gamma').
$$
\end{proof}

\begin{lemma}\label{multLoop} {\sc (Deficiency $2$)} 
Let $\Gamma$ be a deficiency $2$ tropical stable map with non-zero weight. If the loop is adjacent to a trivalent vertex (resp. an edge) and $m$ is the number of interior lattice points on the dual triangle (resp. edge), then $\text{Mult}(\Gamma) = m\cdot\text{Mult}(\Gamma')$, where $\Gamma'$ is the rational curve obtained after contracting the loop.
\end{lemma}
\begin{proof}
Assume first that the loop is adjacent to a $4$-valent vertex. As in the proof of Lemma \ref{multFlat}, we choose a lattice basis for $J$ consisting of two unit vectors for the root vector, a single vector for the edges on either side of the loop (as they are parallel and have equal length) and a unit vector for any other edge of the graph. By definition, the multiplicity of $\Gamma$ is $\frac{m}{2}\cdot\det A$, where $A$ is the matrix representing $J$ in the chosen basis.  The $j$-invariant row of $A$ consists of a single entry $1$ in the column corresponding to the loop. In addition, since there are no marked points on the loop, every path to a point beyond it passes through both of the edges adjacent to the loop. As in the proof of Lemma~\ref{multFlat}, by dividing this column by $2$ and removing the $j$-invariant row and the row corresponding to the loop, we obtain a matrix $A''$ corresponding to the rational curve $\Gamma'$. This produces the formula 
\[
\text{Mult}(\Gamma) = \frac{m}{2}\cdot\det A = \frac{m}{2}\cdot 2\cdot \det A'' = m\cdot\text{Mult}(\Gamma').
\]

\noindent The proof when the loop is based at a $5$-valent vertex follows from similar arguments.
\end{proof}

\subsection{Invariance of multiplicity}

We adapt Urakawa's notion of a harmonic morphism between graphs to the polyhedral setting.

\begin{definition}
Let $\phi:P\to Q$ be a map between weighted cone complexes of the same dimension, where the weight function of $P$ is denoted  $w$. Let $C$ be a co-dimension $1$ cell of $P$ mapping surjectively onto a co-dimension 1 cell $C'$ of $Q$. Let $M'$ be a maximal cone adjacent to $C'$, and let $M_1,\ldots M_k$ be the maximal cells adjacent to $C$ that are mapped to $M'$. Then $\phi$ is said to be \textbf{harmonic} of degree $d$ at $C$ if the sum 
$$
\sum_{i=1}^k w(M_i)\deg(\phi|_{M_i}),
$$
does not depend on $M$. 

\end{definition}

\noindent
If $Q$ is connected through codimension $1$ and pure dimensional then being harmonic at every co-dimension one cell implies that the number of elements in the fibers of $\phi$, counted with multiplicity, is a constant function.

\begin{theorem}
The map $J:\mathcal M_{1,N}(\Delta)\to \Sigma^N\times \mathcal M_{1,1}^{\trop}$ is harmonic. In particular, the number of elliptic curves of a fixed $j$-invariant passing through a generic configuration of points  does not depend on the configuration or the $j$-invariant. 
\end{theorem}

\noindent The theorem follows from Theorem~\ref{thmA} combined with the fact that the analogous algebraic count does not depend on the point configuration or the $j$-invariant. Nonetheless we include a tropical proof, which, in conjunction with Theorem~\ref{thmA}, provides a combinatorial proof for the algebraic invariance statement. 

\begin{proof}[Proof (sketch)]
When an element of  $\Sigma^N\times \mathcal M_{1,1}^{\trop}$ corresponds to a point configuration in general position in $\mathbb{R}^2$, its preimage by $J$ is in the interior of maximal cells of $\mathcal M_{1,N}(\Delta)$. As  $J$ is linear in the interior of maximal cells, its determinant is locally constant, and so is the multiplicity. It remains  to check that the multiplicity does not vary when crossing a wall between maximal cells, namely a co-dimension $1$ cell. 
The possible co-dimension $1$ cells are as follows:
\begin{itemize}
\item $\defi{\alpha} = 0$,  $\alpha$ is of genus $1$,  and has one $4$-valent vertex (besides the $3$-valent vertices);
\item $\defi{\alpha} = 1$,  and $\alpha$ has two 4-valent vertices;
\item$\defi{\alpha}=1$,  and $\alpha$ has one 5-valent vertex;
\item $\defi{\alpha}=2$,  and $\alpha$ has three 4-valent vertices;
\item $\defi{\alpha}=2$,  and $\alpha$ has one 5-valent and one 4-valent vertex;
\item $\defi{\alpha}=2$,  and $\alpha$ has one 6-valent vertex.
\end{itemize}
The proof that the multiplicity remains invariant when crossing a wall is almost identical to the proof of~\cite[Theorem 5.1]{KM}, with minor adjustments for our weights.
We deal with the last case in full detail, and leave the remaining cases as an exercise for the reader.

Assume that $\alpha$ corresponds to curves whose vertices are all trivalent apart for a $6$-valent vertex with a loop. We will exhibit a bijection, up to a constant, between resolutions of such a curve and rational curves. The result then follows from the local invariance of rational curves~\cite[Theorem~4.8]{GM}.
Note that a deficiency $0$ or $1$ tropical curve with no contracted edge does not degenerate to a curve with a loop. Therefore, resolutions of the curve in question are obtained by removing the loop, replacing the vertex $v$ with an edge, and either placing a loop back on the edge, or a contracted edge between two crossing edges, see Figure \ref{def2val6}. These types correspond to subdivisions of the polygon $P$ dual to $v$ having at least one inner lattice point. When the loop is at a vertex dual to a triangle, the multiplicity is the number of interior lattice points of the triangle times the multiplicity of the rational curve obtained by removing the loop by~Lemma \ref{multLoop}. If there is a contracted edge between crossing edges, then  the multiplicity equals that of the rational curve times the area of the parallelogram dual to the crossing,  see~\cite[Lemma 4.11]{KM}. 

It is straightforward to check that the contribution of curves coming from each subdivision is equal to the total number of interior lattice points of $P$ times the multiplicity of the rational curve corresponding to the subdivision, as claimed.

\begin{figure}
\begin{tikzpicture}[font=\tiny,scale=.75]

\begin{scope}[yshift = -2cm, xshift = 6cm, scale=0.4] %dual polygon of co-dim 1 curve
\draw (0,0)--(1,0)--(2,2)--(-3,1)--(0,0);
%\node [left] at (-3,1) {$P$};
\node at (.2,.9) {$P$};

\end{scope}

\begin{scope}[yshift = -3cm, xshift = 6cm, scale=0.3] %co-dim 1 curve
\node [left] at (0,0) {$v$};
\draw (0,0)--(-0.5,2.5);
\draw (0,0)--(-1,-3);
\draw (0,0)--(0,-3);
\draw (0,0)--(4,-2);

\path (0,0) edge[out=90, in=20
                , looseness=.2, loop, dashed,
                 distance=3cm]
            (0,0);
\end{scope}

\begin{scope}[yshift = -6cm, xshift = 0cm, scale=0.4] %dual polygon of curve A
\draw (0,0)--(1,0)--(2,2)--(-3,1)--(0,0);
\draw (0,0)--(2,2);
\end{scope}

\begin{scope}[yshift = -8cm, xshift = -0cm, scale=0.5] %curves A
\path (0,0) edge[out=160, in=90
                , looseness=.2, loop, dashed,
                 distance=2cm]
            (0,0);
\path (1,-1) edge[out=110, in=20
                , looseness=0.6, loop, dashed,
                 distance=2cm]
            (1,-1);           
\path (2,-2) edge[out=90, in=0
                , looseness=0.6, loop, dashed,
                 distance=2cm]
            (2,-2);
            
\draw (0,0)--(-0.5,2.5);
\draw (0,0)--(-1,-3);
\draw (0,0)--(2,-2);
\draw (2,-2)--(2,-3);
\draw (2,-2)--(4,-3);
\end{scope}

\begin{scope}[yshift = -6cm, xshift = 6cm, scale=0.4] %dual polygon of curve B
\draw (0,0)--(1,0)--(2,2)--(-3,1)--(0,0);
\draw (1,0)--(-3,1);
\end{scope}

\begin{scope}[yshift = -8cm, xshift = 6cm, scale=0.5] %curves B
\path (0,0) edge[out=160, in=90
                , looseness=.2, loop, dashed,
                 distance=2cm]
            (0,0);
\path (-.25,-1) edge[out=180, in=110
                , looseness=0.6, loop, dashed,
                 distance=2cm]
            (-.25,-1);           
\path (-.5,-2) edge[out=200, in=130
                , looseness=0.6, loop, dashed,
                 distance=2cm]
            (-.5,-2);
            
\draw (0,0)--(-0.5,2.5);
\draw (0,0)--(2,-1);
\draw (0,0)--(-.5,-2);
\draw (-.5,-2)--(-.5,-3);
\draw (-.5,-2)--(-.75,-3);
\end{scope}

\begin{scope}[yshift = -6cm, xshift = 12cm, scale=0.4] %dual polygon of curve C
\draw (0,0)--(1,0)--(2,2)--(-3,1)--(0,0);
\draw (1,0)--(-2,1);
\draw (-2,1)--(-3,1);
\draw (-2,1)--(2,2);
\end{scope}

\begin{scope}[yshift = -9cm, xshift = 12cm, scale=0.5] %curves C
            
\draw (0,0)--(-.5,2);
\draw (0,0)--(-1,-3);
\draw (0,0)--(2,-1);
\draw (-.5,2)--(-1,4.5);
\draw (-.5,2)--(-.5,-3);

\path (-.5,2) edge[out=80, in=0
                , looseness=.2, loop, dashed,
                 distance=2cm]
            (-.5,2);
\path (0,0) edge[out=80, in=0
                , looseness=0.6, loop, dashed,
                 distance=2cm]
            (0,0);         
              
\path (-.25,1) edge[out=80, in=0
                , looseness=0.6, loop, dashed,
                 distance=2cm]
            (-0.25,1);

\draw (-.2,-.6) [dashed] arc (-10:-53:0.8 and 2.8);
%\draw (-.33,-.9) -- (-.5,-2.5);

\end{scope}

\end{tikzpicture}
\caption{Local picture of a deficiency 2 tropical curve with a 6-valent vertex , its dual polygon, and possible resolutions.}
\label{def2val6}
\end{figure}
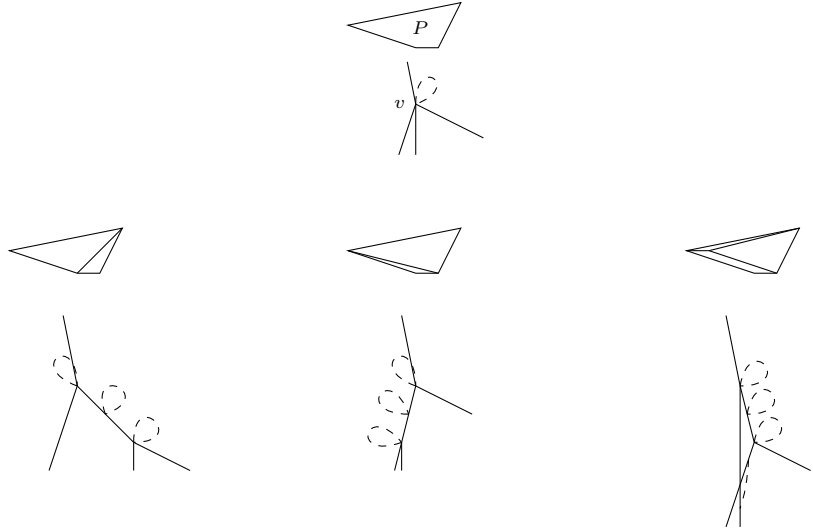
\end{proof}

\section{Logarithmic stable maps and tropical curves}

Our approach to the proof of classical/tropical correspondence is based on the Abramovich--Chen--Gross--Siebert~\cite{AC11,Che10,GS13} theory of logarithmic stable maps. The analysis we carry out is in spirit, closely related to that of Nishinou and Siebert~\cite{NS06} for rational curves, but is modeled on more recent approaches to correspondence theorems using Berkovich spaces and logarithmic geometry~\cite{CMR14a,R15b}. 

We refer the reader to K. Kato's seminal article~\cite{Kat89} and the surveys~\cite{ACGHOSS} for background on logarithmic geometry. We also require certain notions from the theory of Berkovich spaces, in particular the construction of a skeleton of a toroidal embedding. We refer the reader to~\cite{ACMUW,Uli13} for these details. The analytifications of all moduli spaces appearing will be taken over $\CC$, endowed with the trivial valuation. Note however that by definition, the points of an analytified moduli space over such a field will parametrize objects over nontrivially valued field extensions of $\CC$. In particular, the tropicalizations of curves will have nontrivial topology. A reader who is primarily concerned with the enumerative geometry of Hirzebruch surfaces may skip this section, taking Theorem~\ref{thmA} as a black box. 

\subsection{Overview}   Continue to fix a toric surface $X = X(\Sigma)$ with polarization $\Delta$. Intuitively, each tropical stable map $\Gamma \to \Sigma$ is meant to encode a degeneration of a one-parameter family of logarithmic stable maps to $X(\Delta)$. Each tropical multiplicity then encodes the number of ways in which these degenerate curves smooth to the main component of the space of logarithmic maps to $X$. We then calculate the number of elliptic curves by counting tropical maps, weighted by the appropriate combinatorial multiplicities. The subtleties in carrying out  are twofold:

\begin{enumerate}[(A)]
\item Not every tropical stable map arises as a degeneration of a $1$-parameter family of algebraic maps, i.e. not every degenerate map smooths.
\item A single tropical stable map can encode degenerations of logarithmic stable maps into distinct degenerate loci. 
\end{enumerate}

The phenomenon (A) is related to logarithmic obstructedness of stable maps, which manifests tropically as \textit{superabundance}, while (B) is related to the weights of combinatorial types.

\subsection{Logarithmic stable maps} We provide a rapid working introduction  to the basic ideas in logarithmic Gromov--Witten theory, in our setting, referring the reader to~\cite{AC11,Che10,GS13}, where the theory is fully developed. A \textit{logarithmic stable map to $X$ over a logarithmic scheme $(S,M_S)$} is a diagram in the category of fine and saturated logarithmic schemes
\[
\begin{tikzcd}
(C,M_C) \arrow{d}\arrow{r} & (X, M_X) \\
(S,M_S) & 
\end{tikzcd}
\]
such that the underlying map of schemes is a family of ordinary stable maps. In the non-degenerate setting, such a diagram is equivalent to a proper map of pairs
\[
(C,p_1,\ldots, p_k)\to (X,\partial X),
\]
whose image is disjoint from the codimension $2$ toric strata of $X$, such that the orders of contact of $C$ with the boundary along each marked point $p_i$ is locally constant. Here, the logarithmic structure is superfluous to the scheme theoretic information. However, when this data degenerates, for instance, if the contact orders become more degenerate, or if the curve falls into the toric boundary, the logarithmic structure encodes nontrivial information which allows one to interpret contact orders appropriately in this setting. 

One may form a moduli stack over the category of fine and saturated logarithmic schemes by mimicking the usual definition above. However, in order to work geometrically with the space, it is necessary to form a stack over the category of schemes. In other words, given a map from a test scheme $S$ to the moduli space, one needs a universal way to place a logarithmic structure on $S$. This is known as the \textit{minimal logarithmic} structure, and is an explicit condition on the characteristic monoids of the base of a family $(S,M_S)$: at every geometric point $\overline s$ of $S$, the dual monoid of the charcteristic at $\overline s$ is isomorphic to the appropriate cone of tropical curves, determined by the fiber. Minimality is discussed in detail in all of the references given on the subject, and the reader may find a working definition in~\cite[Remark 2.1.1 \& Section 2.3]{R16}. We summarize the results of Abramovich--Chen--Gross--Siebert, as follows. 

\begin{theorem}
There exists a proper moduli stack $\mathscr L_{1,N}(\Delta)$, with a logarithmic structure, over the category of schemes parametrizing minimal logarithmic stable maps $[C\to X]$ with fixed numerical data (contact orders, genus, curve class, and degree). 
\end{theorem}

\subsection{Tropical curves from logarithmic stable maps} Let $\mathscr L_{1,N}(\Delta)$ be the moduli space of minimal logarithmic stable maps from curves of arithmetic genus $1$ to $X$. We assume that the curves meet the boundary of $\Delta$ transversely at labeled marked points. There are $N$ additional marked points having contact order $0$, i.e. over the interior of the main component of $\mathscr L_{1,N}(\Delta)$ these points map to the dense torus. Let $\spec(\CC)\to \mathscr L_{1,N}(\Delta)$ be a point of the moduli space. By pulling back the universal curve, map, and minimal logarithmic structure, we obtain the following diagram
\[
\begin{tikzcd}
\mathscr C \arrow[bend left]{rr} \arrow{d} \arrow{r} & \mathscr U \arrow{d} \arrow{r} & X \\
\spec(P\to \CC) \arrow{r} & \mathscr L_{1,N}(\Delta). & \phantom{123}\\
\end{tikzcd}
\]
\noindent
Choose a monoid homomorphism $P\to \NN$, and consider the induced pull back from the above diagram giving rise to a logarithmic stable map $[f: C\to X]$ over $\spec(\NN\to \CC)$. This data gives rise to a \textit{tropical} stable map, i.e. an integer point in $\mathcal M_{1,N}(\Delta)$, the \textit{tropicalization of $f$}, which we now remind the reader of. It will suffice for our purposes to consider curves that have integer edge lengths and vertex coordinates, so we restrict to this case. \smallskip

\noindent
{\sc Source Graph.} Let $\Gamma$ be the dual graph of $C$, consisting of a vertex for each irreducible component of $C$, and an edge between two vertices for each node that the corresponding components share. Labeled infinite edges are placed at vertices in correspondence with marked points on the corresponding components. The logarithmic structure determines lengths on the edges. Let $e$ be an edge corresponding to a node $q$. Set the length $\ell(e)$ equal to the element in $\NN$ corresponding to the smoothing parameter for the node $q$. This determines an abstract tropical curve $\Gamma$ with integer edge lengths. \smallskip

\noindent
{\sc Map to $\Sigma$.} Let $C_v$ be a component of $C$ corresponding to a vertex $v$ of $\Gamma$. Assume that the generic point of $C_v$ is mapped to the torus orbit of $X$ corresponding to a cone $\sigma\in \Sigma$. Since $f$ is a logarithmic map over $\spec(\NN\to \CC)$, the stalk of the characteristic sheaf of $C$ at the generic point is $\NN$. By virtue of $f:C\to X$ being a logarithmic map, we obtain a homomorphism $S_\sigma\to \NN$, where $S_\sigma$ is the character lattice of the relevant torus orbit. This homomorphism is equivalent to the choice of a lattice element $f^\trop(v)\in \sigma$. 

Let $e$ be an edge adjacent to vertices $u$ and $v$, and $q$ the corresponding node. The stalk of the characteristic sheaf of $C$ at $q$ is given by the monoid push-out $P_q = \NN\oplus_\NN \NN^2$, where the map from $\NN$ to the first factor is the homothety given by multiplication by $\ell(e)$. Assume that $q$ maps to a torus orbit associated to $\sigma$ and let $M_q$ be the corresponding character lattice. Since $f$ is a logarithmic map, we obtain a map $S_\sigma\to P_q$. By the discussion in~\cite[Section 1.4]{GS13}, the data of such a map is \textit{equivalent} to the choices of $f^\trop(u)$, $f^\trop(v)$, and a natural number $c_q$ such that $f^\trop(u)-f^\trop(v) = c_qe_q$. This determines a map on the edge $e$ with expansion factor (weight) equal to $c_q$.

In particular, let $R = \CC[\![t]\!]$ and consider a map 
\[
\spec(R)\to \mathscr L_{1,N}(\Delta),
\] 
the special point is naturally isomorphic to $\spec(\NN\to \CC)$, and thus, gives rise to a tropical map as above. However, not every $\spec(\NN\to \CC)$-valued point of $\mathscr L_{1,N}(\Delta)$ arises in this fashion, because the moduli space is not logarithmically smooth. This brings us to the notions of superabundance and well-spacedness.

\subsection{Superabundant curves and functorial tropicalization} We recall Speyer's \textit{well-spacedness} condition for parametrized tropical genus $1$ curve in $\RR^2$. Let $f: \Gamma\to \RR^2$ be a parametrized tropical curve. If the image of the cycle $\Gamma$ is not contained in any proper affine subspace $\RR^2$ then $\Gamma\to \RR^n$ is realizable. Otherwise the curve is said to be \textit{superabundant}. Let $A$ be an affine line containing the cycle of $f(\Gamma)$. Let $W$ be the multiset of lattice distances from the image of the cycle to points where the cycle leaves the plane.   The curve $\Gamma\to \RR^n$ is said to be \textit{well-spaced} when the minimum of the elements of $W$ occurs at least twice for any affine line $A$ as above.  

All well-spaced curves of genus $1$ arise as degenerations of one parameter families of logarithmic stable maps, from a family of curves with smooth generic fiber. In all the cases considered here, it is known that all such families give rise to well-spaced curves as well. The sufficiency of well-spacedness was follows from work of Speyer~\cite[Theorem 3.4]{Spe14}. In equicharacteristic $0$, and for the superabundant curves appearing in this paper, the necessity of well-spacedness can also be deduced from Speyer's arguments in~\cite[Proof of Proposition 9.2]{Spe14} or directly from Katz' techniques~\cite[Proposition 4.1]{Kat12}. See also work of Nishinou~\cite[Theorems 45 and 52]{Nis10} and Tyomkin~\cite[Theorem~6.2]{Tyo12}. 

For our purposes, the calculations in loc. cit. are sufficient, but in order to the verify the statements, the reader would have to carefully read the proofs and extract results that are not explicitly stated. To avoid this, we choose to prove Theorem~\ref{thmA}  by placing these results on well-spacedness in the context of maps to the Artin fan and the functorial tropicalization results of~\cite{R16}, which we briefly recall. The reader is informed that Corollary~\ref{cor: polyhedral-domain} is the necessary ingredient to continue to the proofs of the main theorems. 

Let $\mathscr A$ denote the \textit{Artin fan} of $X$, i.e. logarithmic algebraic stack $[X/T]$. Let $\mathscr L_{1,N}(\mathscr A)$ be the moduli space of prestable logarithmic maps to the Artin fan $\mathscr A$ with the same discrete data as the maps to $X$ considered previously. The following result is~\cite[Theorem 2.6.2]{R16} stated in the present context, and establishes the relationship between the tropical moduli space of maps and the skeleton of the space of maps to the Artin fan.

\begin{theorem}
The stack $\mathscr L_{1,N}(\mathscr A)$ is toroidal and there is a commutative diagram of continuous morphisms
\[
\begin{tikzcd}
\mathscr L_{1,N}(\mathscr A)^\beth\arrow{rr}{\trop}\arrow[swap]{dr}{\bm p_{\mathscr L}} & & \mathcal M'_{1,N}(\Delta)\\
&\overline{\mathfrak{S}}(\mathscr L_{1,N}(\mathscr A))  \arrow[swap]{ur}{\trop_\Sigma}, &
\end{tikzcd}
\]
where (1) the map $\bm p_{\mathscr L}$ is the projection to the toroidal skeleton and (2) the map
\[
\trop_\Sigma:\overline{\mathfrak{S}}({\mathscr L}_{1,N}(\mathscr A))\to {\mathcal M}'_{1,N}(\Delta)
\]
is a finite morphism of generalized extended cone complexes that is an isomorphism upon restriction to any cell of the source.
\end{theorem}

There is a natural map $\mathscr L_{1,N}(\Delta)\to \mathscr L_{1,N}(\mathscr A)$, and by applying the formal fiber functor $(\cdot)^\beth$ and composing with the map $\trop$ above, we obtain a continuous tropicalization
\[
\trop: \mathscr L_{1,N}(\Delta)^\beth\to \mathcal M'_{1,N}(\Delta). 
\]

We may now rephrase the discussion of superabundant tropical curves as follows. 

\begin{proposition}
The image of the map
\[
\trop: \mathscr L_{1,N}(\Delta)^\beth\to \mathcal M'_{1,N}(\Delta)
\]
is the locus $\mathcal M_{1,N}(\Delta)$ of well-spaced tropical curves.
\end{proposition}

We require a slightly more refined version of this result, stated below. Fix a logarithmic stable map $[f:C\to X]$ of combinatorial type $\Theta = [\Gamma\to \Sigma]$. Let $[\overline f:C\to \mathscr A]$ be the induced map to the Artin fan. Let $\sigma_\Theta$ be the cone of tropical curves of type $\Theta$. Let $U_{\sigma_\Theta}$ be the associated affine toric variety. As explained in~\cite[Section 2]{R16}, the toroidal structure of $\mathscr L_{1,N}(\mathscr A)$ guarantees the existence of toric charts. In other words, there is a smooth neighborhood $U$ of $[\overline f]$  and an \'etale morphism
\[
U\to U_{\sigma_\Theta}. 
\]
Thus, there are natural local monomial coordinates on $U$ corresponding to the deformation parameters of the nodes of $C$, which are in turn, in bijection with the edges of $\Gamma$. Thus, any path in $P$ of edges $\Gamma$ determines a monomial function (up to scalars) in these coordinates, which we will denote $\xi_P$.

\begin{definition}
Let $\Theta$ be a superabundant combinatorial type in $\mathcal M'_{1,N}(\Delta)$ and let $A$ be an affine line containing the loop $L$. Let $P_1,\ldots, P_k$ be the paths from $L$ to those vertices of $\Gamma$, where $L$ leaves the affine hyperplane. The \textbf{well-spacedness function of $(\Theta,A)$} is defined to be the sum
\[
\xi_{P_1}+\cdots +\xi_{P_k}.
\]
If a type is non-superabundant, the well-spacedness function is defined to be $0$. 
\end{definition}

In particular, note that the bend locus in $\sigma_\Theta$ of the tropicalization of a well-spacedness function cuts out the locus of well-spaced curves. Note also that in all the combinatorial types outlined in Section~\ref{sec: moduli-maps}, there is a unique affine line $A$ containing the loop that imposes a nontrivial condition, so each combinatorial type determines at most one nonzero well-spacedness function.  After passing to smooth covers and working locally, the moduli space $\mathscr L_{1,N}(\Delta)$ can be cut out of $\mathscr L_{1,N}(\Delta)$ by explicit equations. This brings us to the next proposition, which describes this precisely.

\begin{proposition}\label{localequations}
Let $[f:C\to X]$ be a minimal logarithmic stable map and let $[\overline f: C\to \mathscr A]$ be the associated map to the Artin fan. Let $\zeta_\Theta$ be the well-spacedness function in a neighborhood $U$ of $[\overline f]$ and let $V = V(\zeta_\Theta)$ be its vanishing locus. Then the base change morphism $\pi$ below
\[
\begin{tikzcd}
\mathscr L_V \arrow[swap]{d}{\pi} \arrow{r} & \mathscr L_{1,N}(\Delta) \arrow{d} \\
V \arrow{r} & \mathscr L_{1,N}(\mathscr A)
\end{tikzcd}
\]
is smooth.
\end{proposition}

\begin{proof}
The result can be proved by mimicking a calculation of Hu and Li for maps to projective space~\cite[Section 4]{HL10}, but for the benefit of the reader we outline how to deduce the necessary calculation from loc. cit, with careful references to the literature.  If $\Theta$ is non-superabundant, then $\mathscr L_{1,N}(\Delta)$ is smooth near $[f]$. To see this, consider the logarithmic tangent-obstruction complex for $[f]$
\[
\cdots \to H^1(C,T_{C}^{\mathrm{log}}) \to H^1(C,f^\star T_{X}^{\mathrm{log}})\to \mathrm{Ob}([f])\to 0.
\]
Note that the first term encodes the infinitesimal logarithmic deformations of $C$, the second encodes obstructions to deforming the map fixing the curve, and the final term is the absolute obstructions of the map and curve. By~\cite[Section 4]{CFPU15}, the non-superabundance of $\Theta$ means that the first arrow is surjective, and by exactness the obstruction group vanishes, so the moduli space is logarithmically smooth in a neighborhood of $[f]$. Thus, we henceforth assume $\Theta$ to be superabundant. 

Our first step is a reduction of dimension for the target. Consider a logarithmic stable map $[f]$ as in the statement of the proposition and let $A$ be the affine space containing the loop. The affine space $A$ determines a subtorus $\mathbb G_m(A)$ of the dense torus of $X$. After replacing $X$ with a toric modification, the quotient by $\mathbb G_m(A)$ induces a morphism
\[
X\to \PP^1,
\]
and a logarithmic prestable map $[g:C\to \PP^1]$. Note that such a toric modification does not change the logarithmic deformation theory of $[f]$ by the results of~\cite{AW13}. Let $\mathrm{Def}(f)$ and $\mathrm{Def}(g)$ denote the logarithmic deformation spaces of $f$ and $g$ respectively. First, we note that the map $\mathrm{Def}(f)\to \mathrm{Def}(g)$ is smooth. To see this, observe that by projecting onto any subtorus complementary to $\mathbb G_m(A)$, the loop is not contracted, and smoothness follows from well-known calculations, see~\cite{CFPU15} or~\cite[Section 2]{HL10}. Thus, it suffices to prove the proposition with $X$ replaced by $\PP^1$. 

To complete the proof, we consider the logarithmic map $[g: C\to \PP^1]$. Let $\mathscr X$ be a chain of $\PP^1$'s and consider a logarithmic stable map to $[g':C'\to \mathscr X]$, such that (1) no component of $C'$ is contracted to a node of $\mathscr X$, and (2) the composition $\pi\circ g'$ is the logarithmic stable map $[g]$. That the condition (1) and (2) can be achieved is a standard fact in relative Gromov--Witten theory, see for instance~\cite{NS06}. The logarithmic deformation theory of $[g]$ and $[g']$ coincide via the contraction to the main component $\mathscr X\to \PP^1$, by the results of~\cite{AMW}. Since $\Theta$ was superabundant, there is a subcurve of $C'$ of arithmetic genus $1$ that is contracted to a smooth point of the expanded target $\mathscr X$. Since the complement of this contracted subcurve is a union of curves of arithmetic genus $0$ the obstructions to deformations are local on the target $\mathscr X$, near this point to which the subcurve is contracted. The statement is now reduced to~\cite[Corollary 4.14]{HL10}. 
\end{proof}

In particular, we require the following corollary.

\begin{corollary}\label{cor: polyhedral-domain}
Let $\sigma\to \mathcal M_{1,N}(\Delta)$ be a cone of well-spaced tropical curves. The preimage of $\sigma$ under the tropicalization map
\[
\trop: \mathscr L_{1,N}(\Delta)^\beth\to \mathcal M_{1,N}(\Delta)
\]
is an analytic domain of the source. 
\end{corollary}

\begin{proof}
The cone $\sigma$ is contained in the relative interior of a unique smallest cone of $\mathcal{M}'_{1,N}(\Delta)$, which we denote by $\sigma_{\Theta}$. Let $[f]$ be a minimal logarithmic stable map to $X$ with combinatorial type $\sigma'$ and let $[\overline f]$ be the associated map to the Artin fan. Let $U_{\sigma_\Theta}$ be the toric neighborhood (in the smooth topology) of $[\overline f]$ as discussed above. In all the combinatorial types relevant to the enumerative problem, the well-spacedness equation $\zeta_\Theta$ is a binomial equation, since it forces the lengths of two paths to be equal to each other. In particular, if $\zeta_\Theta$ is the well-spacedness function associated to $\Theta$, the vanishing locus $V(\zeta_\Theta)$ is the closure of a subtorus in $U_{\sigma_\Theta}$. Thus, applying the preceding proposition and working \'etale locally, $\trop^{-1}(\sigma)$ is the preimage of a cone under tropicalization of a toric variety, and thus is, essentially by definition, a polyhedral analytic domain, in the sense of~\cite{Rab12}. 
\end{proof}

%\footnote{When all edge multiplicities are equal to $1$, this deduction is immediate from Speyer's proof. When the edges have higher weight, one may first use the branched covering trick~\cite[Lemma 5.2]{NS06}, recombine the local piece to the degeneration and reduce to edge weight $1$. Then one can apply Speyer's result.}

\subsection{Proof of Theorem~\ref{thmA}} As we have seen previously, every logarithmic stable map comes with a tropical type. In order to prove the main theorem, we will have to reverse this, and determine how many logarithmic lifts there exist for a given tropical type. To guide the reader in the proof, we point out that the multiplicity of a tropical curve in the previous section was a weight factor, multiplied by a determinant of a map of lattices. In the analysis to come, this weight factor answers the question of how many logarithmic lifts there exist of a tropical curve, while the determinant will arise as an analytically local contribution to the degree of map.

To analyze the weights, we need the notion of an unsaturated map. Let $\xi = [f: C\to X]$ be a minimal logarithmic stable map over $\spec(P\to \CC)$. As explained in~\cite[Section 3.6]{R15b}, one may associate to this another logarithmic stable map $\xi^{\mathrm{us}}$, the \textit{unsaturated map}, over $\spec(Q\to \CC)$, in the category of fine but \textit{not necessarily saturated} logarithmic structures. The monoid $Q$ need not be saturated, but its saturation is $P$. We refer to loc. cit. for the precise construction, but record that the unsaturated morphism has the following important property~\cite[Proposition 3.6.3]{R15b}. 

\begin{proposition}\label{saturation}
Let $\xi_1$ and $\xi_2$ be two minimal logarithmic stable maps over the same base such that the underlying ordinary stable maps $\underline \xi_1$ and $\underline \xi_2$ coincide, and the combinatorial types of the tropical maps associated to $\xi_1$ and $\xi_2$ also coincide. Then there is a canonical isomorphism $\xi_1^{\mathrm{us}}\cong \xi_2^{\mathrm{us}}$. 
\end{proposition}

%We prove the theorem by considering logarithmic stable maps over $\spec(\NN \to \CC)$, choosing point conditions whose tropicalizations are in general position. We then count tropical curves meeting these points, and analyze lifts over $\spec(\NN\to \CC)$ that smooth to the main component of $\mathscr L_{1,N}(\Delta)$. 
%
%Consider a logarithmic stable map $[f: C\to X]$ over $\spec(\NN\to \CC)$. The curve $C$ has marked points $p_1,\ldots, p_N$, which have trivial contact order. Assume $p_i$ maps to the stratum $V(\sigma)$ and let $S_\sigma$ denote the character lattice of the dense torus in this stratum. The stalk of the relative characteristic of $C$ at $p_i$ is $\NN$, so $[f]$ determines a map $S_\sigma\to \NN$. Dualizing, this produces a point $p_i^{\trop}$ of $N$. Choosing $p_i^{\trop}$ to be in general position and the tropical $j$-invariant of $C^{\trop}$ to be positive, it follows that $C$ must a nodal union of genus $0$ curves. Fixing the $j$-invariant of $C$ amounts to forcing the stable map $f$ to lie in a fiber of the forgetful morphism $\mathscr L_{1,N}(\Delta) \to \overline{\mathcal{M}}_{1,1}$. 

We prove the theorem by considering the natural morphism $\mathscr L_{1,N}(\Delta)\to X^N \times \overline{\mathcal{M}}_{1,1}$, evaluating the positions of the $N$ marked points, and recording the $j$-invariant. The enumerative invariant we are interested in is the number of elements in the fiber of this map that smooth to the main component of $\mathscr L_{1,N}(\Delta)$, counted with multiplicity. The morphism $\mathscr L_{1,N}(\Delta)\to X^N \times \overline{\mathcal{M}}_{1,1}$ is logarithmic and by~\cite[Theorem D]{R16} there is an induced morphism
\[
J: {\mathcal M}_{1,N}(\Delta)^\trop \to \Sigma^N\times \mathcal M_{1,1}^\trop.
\]

%Let $\mathscr A_\Delta = [X(\Delta)/T]$ be the Artin fan associated to the toric variety $X(\Delta)$, in the sense of~\cite{AW13,R16}. It follows from~\cite[Section 2.6]{R16}, that this induced map $\bm p$ factors as
%\[
%\bm p:\mathscr L^{\beth}_{1,N}(\Delta) \xrightarrow{\pi_1} \mathscr L^{\beth}_{1,N}(\mathscr A_\Delta) \xrightarrow{\pi_2} \overline{\mathcal M}'_{1,N}(\Delta)^\trop,
%\]
%where the first map is the map induced by canonical projection $X(\Delta)\to \mathscr A_\Delta$, and the second map is a retraction of a toroidal embedding onto its skeleton. Recall here that $\overline{\mathcal M}'_{1,N}(\Delta)^\trop$ was the set of all parametrized tropical curves, without well-spacedness. By~\cite{R16}, the image of $\mathscr L^\beth_{1,N}(\Delta)$ is the well-spaced sublocus $\overline{\mathcal M}'_{1,N}(\Delta)^\trop$.

Fix a combinatorial type $\Theta$ and a single minimal logarithmic stable $[f]$ with combinatorial type $\Theta$. Let $\sigma$ be a cone of well-spaced tropical stable maps in ${\mathcal M}_{1,N}(\Delta)^\trop$ parametrizing maps with type $\Theta$. By Corollary~\ref{cor: polyhedral-domain}, $\trop^{-1}(\sigma)$ is a polyhedral domain $U^\beth$ in $\mathscr L^{\beth}_{1,N}(\Delta)$. The cone $\sigma$ maps to a cone $\tau$ in $\Sigma^N\times \mathcal{M}_{1,1}^\trop$. The preimage of $\tau$ under tropicalization in $(X\times \overline{\mathcal M}_{1,1})^\beth$ gives a compact polyedral domain $V^\beth$, containing $J^\beth([f])$. The contribution of $[f]$ of our enumerative count is the degree of the resulting morphism of analytic domains
\[
J^\beth: U^\beth \to V^\beth.
\]
Note that since we work with stacks, this degree is computed by passing to Galois covers to kill automorphisms, computing the degree on these covers, and dividing out by automorphisms. After passing to such a cover, this is a map of polyhedral domains, by~\cite[Section 6]{Rab12}, so the degree of this morphism is equal to the determinant of the map of cones $\sigma\to\tau$. 

To establish the result, it remains to add the contributions over the preimages of all analytic domains obtained over cones contributing to our enumerative invariant. More precisely, given a combinatorial type $\Theta$, we must find the number of minimal logarithmic stable maps of type $\Theta$, with the same underlying stable map. By Proposition~\ref{saturation} above, this is the same as the degree of the saturation map for the type $\Theta$. Each saturated map will give one analytic domain in the corresponding type, and since the map $\trop_\Sigma$ of the previous subsection is an isomorphism on every face of the skeleton of $\mathscr L_{1,N}(\mathscr A)$, different saturations of the same map yield the same contribution to the degree of $J$. Our task is to show that the degree of this saturation map is equal to the weight in Definition~\ref{weights}. Summing these determinants over all the combinatorial types, the result will follow. We now analyze the weights, based on the deficiency.  \smallskip

\noindent
{\sc Deficiency $2$.} We need to relate the multiplicities of the genus $1$ maps to those of genus $0$ maps, and apply Lemma~\ref{multLoop}. Consider a minimal logarithmic map $[f:C\to X]$ whose tropicalization $[f^\trop:\Gamma\to\Sigma]$ has deficiency $2$, and meets general point and $j$-invariant constraints. The source curve has a single self-nodal component $D$ corresponding to a vertex $v_D$ supporting a loop. By general position considerations, the loop must be adjacent to a $4$ or $5$-valent vertex in the source graph $C^\trop$. When $v_D$ has valence $5$, the arithmetic genus of $f(D)$ is equal to the number of interior lattice points in the triangle dual to the star of $f^\trop(v_D)$ in $f^\trop(\Gamma)$. Call this number $P_D$. Normalize the self node of $D$ to obtain a new map $[\tilde f: \tilde C\to X]$. Since the length of the loop is fixed, the cone of tropical curves associated to $[\tilde f]$ coincides with the one associated to $[f]$. Given a map $[\tilde f: \tilde C\to X]$ from a genus $0$ curve, there exist $P_D$ maps from arithmetic genus $1$ curves, by choosing a node of the image to not separate. This recovers the multiplicity of Lemma~\ref{multLoop}. 

Now assume $v_D$ is attached at the interior of an edge of $C^\trop$. We consider the number of minimal logarithmic stable maps that can have the prescribed tropicalization and compare it to the number of maps from genus $0$ curves. Let $w$ be the weight of the edge to which the loop contracts. The component $D$ has a self node and maps onto its image as $D\to \PP^1$, a $w$-fold cover fully ramified over two points. Normalizing the node, we obtain a map $\PP^1\to D\to \PP^1$, which is a $w$-fold cover. All possible maps from $D$ are formed by choosing two preimages on this cover and gluing them. There are precisely $\frac{w(w-1)}{2}$ ways to make this choice. However, there is an overall action of $\mathbb Z/w\ZZ$ acting on the cover by roots of unity which give automorphisms of the map. To compute the degree of the map of stacks, we simply quotient by this group, and as above, we recover the desired multiplicity.\smallskip

\noindent
{\sc Deficiency $1$.} Consider logarithmic stable maps having a fixed underlying map $[\underline C\to X]$ and fixed combinatorial type with deficiency $1$. This implies that any two unsaturated maps with this data coincide, say with $\xi^{\mathrm{us}}$, and to prove the result we must show that the weight in Definition~\ref{weights} coincides with the number of ways in which to saturate $\xi^{\mathrm{us}}$. The source curve $C$ has two components meeting at precisely two points, giving a double edge to the dual graph. Ignoring other nodes, the logarithmic deformation space for this local geometry is non-normal: it is isomorphic to $\spec(\mathbb C\llbracket x,y\rrbracket/(x^{m_1} = y^{m_2}))$, see for instance~\cite[Section 4.2]{CMR14a} or~\cite[Example 1.17(2)]{GS13}. Here, the parameters $x$ and $y$ are deformation parameters for the two nodes of the source curve. An unwinding of definitions shows that each of these parameters, raised to the ramification orders at the corresponding node, is equal to the pullback of the same monomial function on $X$. The multiplicity of this tropical curve, which is the number of saturations of $\xi^{\mathrm{us}}$, is precisely the number of branches in the normalization of this local geometry. A direct computation then shows that this is in turn equal to $\mathrm{gcd}(m_1,m_2)$, so we recover the weight from Definition~\ref{weights}. \smallskip

\noindent
{\sc Deficiency $0$.} The deficiency $0$ case is very similar to the case above, but it is easier due to the lack of the well-spacedness condition. Let $[f:C\to X]$ be a logarithmic stable map such that $f^\trop$ has deficiency zero and meets general tropical point and $j$-invariant conditions. It follows that the combinatorial type $\alpha$ of $f^\trop$ is trivalent with associated cone $\mathcal{M}_{1,N}(\Delta)^\alpha$. Since $f^\trop$ has deficiency $0$ and lies in $\RR^2$, the cone of tropical maps in the combinatorial type $\alpha$ has expected dimension and in particular, $f^\trop$ is non-superabundant. As explained in the proof of Proposition~\ref{localequations}, this means that $\mathscr L_{1,N}(\Delta)$ is logarithmically smooth at the point $[f]$. Thus, $\mathscr L_{1,N}(\Delta)$ is \'etale locally at $[f]$, isomorphic to the toric variety given by the stalk of the characteristic of $\mathscr L_{1,N}(\Delta)$ at $[f]$. Let $\alpha$ be the combinatorial type of $[f^\trop]$. By~\cite[Remark 1.21]{GS13}, this toric variety is precisely the one defined by the cone $\mathcal M_{1,N}(\Delta)^\alpha$. 

We must compute the number of minimal logarithmic stable maps with a fixed tropicalization and fixed underlying stable map. By the above proposition, the number of logarithmic lifts of $[\underline f]$ that have combinatorial type $\alpha$ is equal to the index the unsaturated characteristic of $[f]$ in the saturation. By dualizing and using~\cite[Remark 1.21]{GS13}, we see that the rank of this saturation coincides with the index of the length constraint equations $\binom{a_1}{a_2}$ associated to the cycle in the type $\alpha$. For each fixed logarithmic lift, its contribution to our enumerative count is equal to the determinant of the induced map $J: \mathcal M_{1,N}(\Delta)\to \Sigma^N\times \mathcal M_{1,1}^\trop$. Thus, the multiplicity of $\alpha$ is equal to the index of $\binom{a_1}{a_2}$ times the determinant of the map $J$. This recovers the multiplicity of Lemma~\ref{multDet}. The result follows.

\qed

\section{Curves on Hirzebruch surfaces}
For this section, fix $X$ to be the Hirzebruch surface $\mathbb F_n = \mathbb P(\mathcal O_{\mathbb P^1}\oplus \mathcal O(n))$. Let $\Sigma$ be its fan and $\Delta$ a polytope polarizing the surface. The Picard group of $X$ is free of rank $2$, generated by a fiber class $f$ and the class $s$ of the unique section having self-intersection $-n$. A curve $C$ has bidegree $(a,b)$ if $C\cdot f = a$ and $C\cdot s = b$. Since $\mathbb F_n$ is toric, this data determines a unique Chow cohomology class in $A^1(\mathbb F_n)$, for instance, by applying Fulton--Sturmfels results on the Chow cohomology of toric varieties~\cite{FS97}. A \textit{tropical curve} $\Gamma$ of bidegree $(a,b)$ on $\mathbb{F}_n$ will be a tropical curve in $|\Sigma|$ whose set of infinite ends is 
$$
\bigg\{a\cdot\binom{n}{1}, b\cdot\binom{1}{0},a\cdot\binom{0}{-1},(an+b)\cdot\binom{-1}{0}\bigg\}.
$$
\noindent
In other words, under the canonical isomorphism between the Minkowski weights of codimension $1$ on $\Sigma_{\mathbb F_n}$ and $A^1(\mathbb F_n)$, the recession fan of $\Gamma$ represents a curve of bidegree $(a,b)$.

\begin{figure}[h!]
\begin{tikzpicture}[font=\tiny,scale=.5]

\begin{scope} %The tropical curve
\draw (0,0)--(2,0)--(2,-1)--(0,-1)--(0,0);
\fill[marked point] (2,-.5) circle (.15);
%\node[above] at (1,0) {$2$};

\draw (2,0)--(4,1);
\draw (2,-1)--(3,-2);
\draw (3,-2)--(4,-2);
\draw (3,-2)--(3,-3);
\fill[marked point] (3,-2.5) circle (.15);
\draw (0,-1)--(-1,-2);
\draw (-1,-2)--(-1,-3);
\fill[marked point] (-1,-2.5) circle (.15);
\draw (-1,-2)--(-2,-2);

\draw (0,0)--(-2,1);

\draw (-2,1)--(-4,1);
\fill[marked point] (-3,1) circle (.15);

\draw (-2,1)--(-3,2);
\fill[marked point] (-2.5,1.5) circle (.15);

\draw (-3,2)--(-4,2);

\draw(-3,2)--(-3,3);
\fill[marked point] (-3,2.5) circle (.15);

\draw(-3,3)--(-4,3);
\fill[marked point] (-3.5,3) circle (.15);

\draw(-3,3)--(-2,4);

\draw(-2,4)--(-4,4);
\fill[marked point] (-3,4) circle (.15);

\draw(-2,4)--(0,5);
\fill[marked point]  (-1,4.5) circle (0.15);
\end{scope}

\begin{scope}[shift={($(10,-2.5)$)},scale=1.85] %The Newton polygon.
\draw (0,0)--(2,0)--(2,1)--(0,5)--(0,0);

\draw (0,1)--(1,0)--(2,1)--(0,1);
\draw(1,3)--(1,0);
\draw(1,3)--(0,1);
\draw(1,3)--(0,2);
\draw(1,3)--(0,3);
\draw(1,3)--(0,4);
\end{scope}

\end{tikzpicture}
\caption{A curve of degree $(2,1)$ through $9$ points on $\mathbb{F}_2$ and its Newton polygon.}
\label{tropicalCurve}
\end{figure}
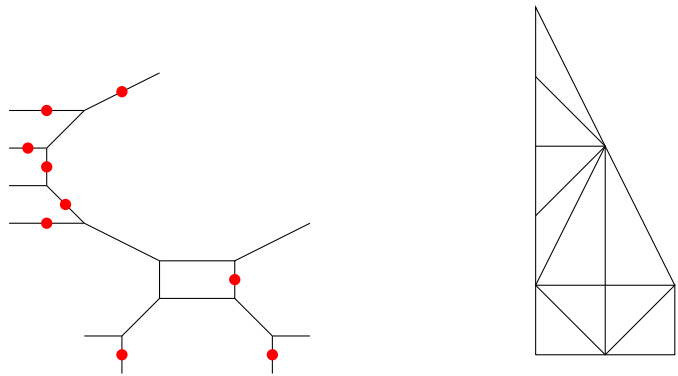

\begin{definition}
A \emph{string} of a tropical curve is a subgraph homeomorphic to either $S^1$ or $\mathbb{R}$, that does not contain any  marked points.
\end{definition}

In what follows, we will show that requiring the $j$-invariant to be very large imposes strong conditions on the combinatorial type of a tropical curve. The formula  in  Theorem \ref{thmB} will then follow by degenerating to a large $j$-invariant and computing the multiplicities of such types.

\begin{lemma}\label{largeJ}
Let $\Gamma$ be a tropical curve with a large $j$-invariant passing through a point configuration $\mathcal{P}$ in general position. Then one of the following is true.

\begin{enumerate}[(1)]
\item\label{contracted} $\Gamma$ has a contracted edge.
%{\centering
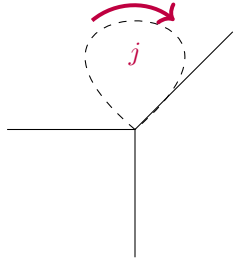
\begin{figure}[h!] %a curve with a contracted loop.
\begin{tikzpicture}[font=\small,scale=1]
\draw (0,0) -- (1.3,1.3);
\draw (0,0) -- (0,-1.7);
\draw (0,0) -- (-1.7,0);
\path (0,0) edge[ out=140, in=40
                , looseness=0.8, loop, dashed
                , distance=3cm]
            (0,0);
\draw[->,purple, line width=1.5pt] (-0.57,1.45) arc (130:50:0.85cm);
\node[purple] at (0,1) {$j$};
\end{tikzpicture}
\caption{A curve with a contracted edge.}
\end{figure}
%}
\item\label{string} $\Gamma$ has a string that can be moved to the right. The string is dual to  a triangle with vertices $(0,0), (1,0),(0,n)$, and $k$ edges emanating from $(1,0)$ towards the opposite edge, and two of the bounded edges emanating from the string are part of the cycle. See Figure \ref{stringCycle}.

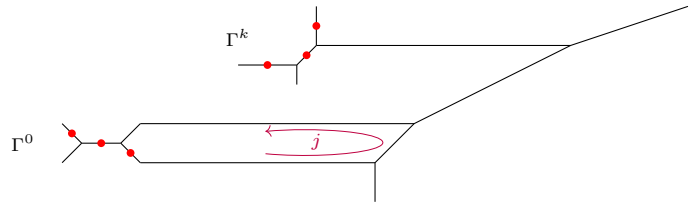
\begin{figure}[h!]% a curve with a string.
{\centering
\begin{tikzpicture}[font=\tiny,scale=.52]%[scale=.3, font=\tiny]

\draw (-2,-1) -- (-2,0) -- (-1,1) -- (3,3) -- (6,4);

%First rational curve.
\draw (3,3) -- (-3.5,3);
\draw (-3.5,3) -- (-3.5,4);
\fill[marked point] (-3.5,3.5) circle (.1);

\draw (-3.5,3) -- (-4,2.5);
\fill[marked point] (-3.75,2.75) circle (.1);

\draw (-4,2.5) -- (-5.5,2.5);
\fill[marked point] (-4.75,2.5) circle (.1);

\draw (-4,2.5) -- (-4,2);

\node at (-5.5,3.2) {$\Gamma^k$};

%Rational curve connected to cycle. 
\draw (-2,0) -- (-8,0);
\draw (-1,1) -- (-8,1);

\draw (-8,0) -- (-8.5,0.5);
\fill[marked point] (-8.25,.25) circle (.1);
\draw (-8,1) -- (-8.5,0.5);
\draw (-8.5,0.5) -- (-9.5,0.5);
\fill[marked point] (-9,.5) circle (.1);
\draw (-9.5,0.5) -- (-10,1);
\fill[marked point] (-9.75,0.75) circle (.1);
\draw (-9.5,0.5) -- (-10,0);

\draw[->,purple, line width=.2pt] (-4.8,.23) arc (-120:120:2 and 0.33);
\node [purple] at (-3.5,.54) {$j$};

\node at (-11,.5) {$\Gamma^0$};

\end{tikzpicture}   

}
\caption{A curve with a string.}\label{stringCycle} 
\end{figure}

\item\label{stringFlat}
The same as the previous case, except that the cycle is flat and adjacent to one of the bounded edges emanating from the string, see Figure~\ref{fig: stringFlat}.

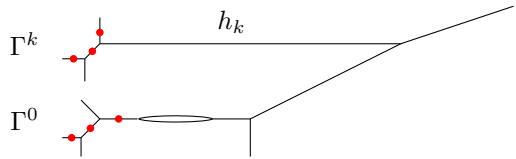
\begin{figure}[h!] % a curve with a string.
{\centering
\begin{tikzpicture}[font=\small,scale=0.5]%[scale=.3, font=\tiny]

%Tropical curve with flat loop.
\draw (0,0) -- (0,1) -- (4,3) -- (7,4);

\draw (0,1) -- (-1,1);
\draw (-1,1) arc (10:170:1 and 0.1);
\draw (-1,1) arc (-10:-170:1 and 0.1);

\draw (-3,1) -- (-4,1);
\draw (-4,1) -- (-4.5,1.5);
\draw (-4,1)--(-4.5,0.5);
\draw (-4.5,0.5)--(-5,0.5);
\draw (-4.5,0.5)--(-4.5,0);

\fill[marked point] (-4.75,0.5) circle (.1);
\fill[marked point] (-4.25,0.75) circle (.1);
\fill[marked point] (-3.5,1) circle (.1);

\node at (-6,1) {$\Gamma^0$};

\draw (4,3) -- (-4,3);
\node [above] at (-0.5,3) {$h_k$};
\draw (-4,3)--(-4,3.7);
\draw (-4,3)--(-4.4,2.6);
\draw (-4.4,2.6)--(-5,2.6);
\draw (-4.4,2.6)--(-4.4,2);

\fill[marked point] (-4,3.3) circle (.1);
\fill[marked point] (-4.2,2.8) circle (.1);
\fill[marked point] (-4.7,2.6) circle (.1);

\node at (-6,3) {$\Gamma^k$};
%\end{scope}
\end{tikzpicture}   

}
\caption{A curve with a string and a flat cycle.}\label{fig: stringFlat}
\end{figure}

\end{enumerate}
\end{lemma}

\begin{proof}
Consider the map $J:\mathcal M_{1,N}(\Delta)\to \mathbb{R}^{2n}\times \mathcal M_{1,1}$. We will show that for cells that do not correspond to one of the types listed above, the intersection of their image with $\{\mathcal{P}\}\times \mathcal M^\trop_{1,1}$ is bounded. Since there are  finitely many maximal dimensional cells, the result follows. Note first, that since the restriction of $\{\mathcal{P}\}\times \mathcal M^\trop_{1,1}$ to  each cell is linear and $\mathcal{P}$  is in general position, the preimage of $\{\mathcal{P}\}\times \mathcal M^\trop_{1,1}$ is one dimensional and connected. 

Assume for the sake of contradiction that there is a cell of $\mathcal M_{1,N}(\Delta)$ which does not correspond to one of the types above, and such that the $j$-invariant is not bounded. Starting from a tropical curve $\Gamma$, we may deform it using the fact that $J^{-1}(\{\mathcal{P}\}\times \mathcal M_{1,1})$ is connected to obtain a higher $j$-invariant, while fixing the points. As shown during the proof of \cite[Proposition 4.49]{Markwig}, the only way to deform a curve while maintaining the point condition is by moving a string. We claim that the curve has precisely one string. Otherwise, $\Gamma$ moves in a two dimensional family, which contradicts our assumption.  Since we assume that the $j$-invariant is unbounded, and the only way to deform the curve  is by moving the string, it cannot be a cycle. Moreover, all the vertices of $\Gamma$ must be on one side of the string. Otherwise, the string cannot move indefinitely without changing combinatorial type. 

Denote the edges of the string $e_1,\ldots,e_k$, and the bounded edges emanating from them $h_1,\ldots,h_{k-1}$. Arguing as in \cite[Lemma 2.10]{FM}, the string and the edges emanating from it are dual to a subdivision of a  polygon whose boundary edges correspond to $e_1,e_k$ and $h_1,\ldots, h_{k-1}$, such that the edges dual to $h_1,\ldots, h_{k-1}$ are concave. For each $i,j$, denote $\hat{e}_i,\hat{h}_j$ the edge of the polygon dual to $e_i,h_j$ (see Figure \ref{stringPolygon1}). As $e_1$ and $e_k$ are ends of $\Gamma$, the edges $\hat{e}_1,\hat{e}_k$ are boundary edges of $\Delta$. 

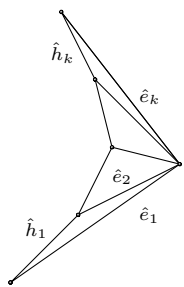
\begin{figure}[h!]
\begin{tikzpicture}[font=\tiny,scale=.75]

\begin{scope}[yshift = -3cm, xshift = 6cm, scale=0.6] 
\draw (-1,0)--(4,3.5)--(0.5,8)--(1.5,6)--(2,4)--(1,2)--(-1,0);

\draw (4,3.5)--(0.5,8);
\draw (4,3.5)--(1.5,6);
\draw (4,3.5)--(2,4);
\draw (4,3.5)--(1,2);

\node [right] at (2.5,2) {$\hat{e}_1$};
\node [above] at (2.3,2.6) {$\hat{e}_2$};
\node [right] at (2.5,5.5) {$\hat{e}_k$};

\node [left] at (1.2,6.8) {$\hat{h}_k$};
\node [left] at (0.4,1.6) {$\hat{h}_1$};

\draw [ball color=black] (-1,0) circle (0.5mm);
\draw [ball color=black] (4,3.5) circle (0.5mm);
\draw [ball color=black] (.5,8) circle (0.5mm);
\draw [ball color=black] (1.5,6) circle (0.5mm);
\draw [ball color=black] (2,4) circle (0.5mm);
\draw [ball color=black] (1,2) circle (0.5mm);

\end{scope}
\end{tikzpicture}
\caption{Polygon with a concave side.}
\label{stringPolygon1}
\end{figure}

We claim that the direction vectors of $\hat{e}_1,\hat{e}_k$ are,  in fact, $(1,0)$ and $(-1,n)$. Suppose otherwise. The polygonal path $\hat{h}_1,\ldots, \hat{h}_{k-1}$ is concave, so the vertices separating the segments are in the triangle spanned by the edges dual to $e_1,e_k$. If those edges are different from $(1,0)$ and $(-1,n)$, the triangle does not contain any integer points, and $k$ must equal $2$. But then there is only a single edge of weight zero emanating from the string, and in particular, the $j$-invariant cannot be changed by moving the string. 

From the discussion, we see that $e_1,e_k$ are in directions $\binom{0}{-1}, \binom{1}{n}$, and the triangle spanned by $\hat{e}_1,\hat{e}_k$ is as shown in the figure below. In particular, $h_1,\ldots,h_{k-1}$ are horizontal. The cycle is formed by either two edges emanating from the string or a flat cycle connected to the string by an edge, and the $j$-invariant varies by translating the string.

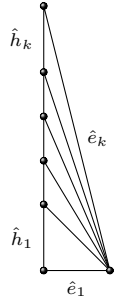
\begin{figure}[h!]
\begin{tikzpicture}[font=\tiny,scale=1.1]

\begin{scope}[yshift = -3cm, xshift = 6cm, scale=0.8] 
\draw (0,0)--(1,0)--(0,4)--(0,0);
\draw (1,0)--(0,1);
\draw (1,0)--(0,1.66);
\draw (1,0)--(0,2.33);
\draw (1,0)--(0,3);

\node [below] at (.5,0) {$\hat{e}_1$};
\node [right] at (0.5,2) {$\hat{e}_k$};

\node [left] at (0,0.5) {$\hat{h}_1$};
\node [left] at (0,3.5) {$\hat{h}_k$};

\draw [ball color=black] (0,0) circle (0.5mm);
\draw [ball color=black] (1,0) circle (0.5mm);

\draw [ball color=black] (0,1) circle (0.5mm);
\draw [ball color=black] (0,1.66) circle (0.5mm);
\draw [ball color=black] (0,2.33) circle (0.5mm);
\draw [ball color=black] (0,3) circle (0.5mm);
\draw [ball color=black] (0,4) circle (0.5mm);
\end{scope}
\end{tikzpicture}
\caption{The polygon dual to the string.}
\label{stringPolygon2}
\end{figure}

\end{proof}

In what follows, we will make use of rational curves in a very general position, namely \textit{simple curves}~\cite[Definition 4.2]{Mikhalkin}.
\begin{definition}\label{simple}
A tropical curve is called \emph{simple} if its dual subdivision contains only triangles and parallelograms.
\end{definition}
\noindent Given a map from a tropical elliptic curve of type \ref{contracted} in Lemma \ref{largeJ}, we may remove the contracted edge to obtain a rational curve passing through the same point configuration. 
When the curve is simple, the number of ways of doing this, counted with multiplicity, has a simple combinatorial description. 
For a lattice polygon $P$ in $\RR^2$ we denote its number of interior lattice points by $\overset{\circ}\# P$.

\begin{proposition}\label{prop:contracted}
Let $\Gamma'$ be a simple rational curve passing through a configuration of $N$ points in general position, and let $L$ be  a large real number (in the sense of Lemma \ref{largeJ}). Then the sum of the multiplicities of elliptic curves with $j$-invariant $L$, obtained from $\Gamma'$ by adding a contracted bounded edge, is
$$
\overset{\circ}{\#}\Delta\cdot\text{Mult}(\Gamma').
$$
\end{proposition}
\begin{proof}
The proof is an adaptation of Lemma 6.2 in \cite{KM} with the modified weights used here. Let $\Gamma$ be an elliptic curve obtained from $\Gamma'$ by adding an edge.  By the balancing condition, together with the fact that $\Gamma'$ is simple, such an edge can either connect two crossing edges of $\Gamma'$ or be the loop. 

In the first case, the crossing edges are dual to a parallelogram $P$ in $\Delta$, and~\cite[Lemma 4.11]{KM} implies that the multiplicity of $\Gamma$ equals the multiplicity of $\Gamma'$ multiplied by the area of $P$. By Pick's theorem \cite{Pick}, the area of $P$ is
$$
A(P) = \overset{\circ}{\#}(P) + \frac{b}{2} + 1,
$$
where $b$ is the number of lattice points in the interior of the edges of $P$. 

Now, assume that $\Gamma$ has a contracted loop. If the loop is adjacent to a $5$-valent vertex of $\Gamma'$, then Lemma \ref{multLoop} shows that the multiplicity of $\Gamma$ is $\text{Mult}(\Gamma')$ times the number of interior lattice points of the triangle dual to the vertex. 
Finally, if the loop is attached to an edge of $\Gamma'$, the multiplicity equals the number of interior lattice points on the dual edge of the Newton polygon times $\text{Mult}(\Gamma')$. 
In conclusion, the total sum of the multiplicities of all the possible elliptic curves $\Gamma$ giving rise to $\Gamma'$ is
$$
\sum_{P} (\overset{\circ}{\#}(P) + \frac{b}{2} + 1) + \sum_T \overset{\circ}{\#}(T) + \sum_E\overset{\circ}{\#}(E),
$$
where the first and second sums are taken over all the parallelograms and  triangles in the Newton subdivision of $\Delta$ respectively. The third sum is taken over the interior edges in the Newton subdivision, where we identify opposite edges of a parallelogram. 
 Denote $I_\Delta$ the number of interior lattice points of $\Delta$ that appear as vertices of the subdivision. 
Since $\Gamma'$ is simple, and its genus is zero, the number of parallelograms in the subdivision equals $I_\Delta$. The above sum becomes
$$
\sum_{P}( \overset{\circ}{\#}(P) +\frac{b}{2}) + \sum_T \overset{\circ}{\#}(T) + \sum_E\overset{\circ}{\#}(E) + I_\Delta.
$$
Every interior lattice points of $\Delta$ appears in the sum exactly once, therefore, it equals $\overset{\circ}{\#} \Delta$.

\end{proof}

\subsection{Rational curves with contact orders}
Our argument below will require us to consider rational curves with certain tangency conditions with the boundary. 
Therefore, we make the following definition:

\begin{definition}\label{relativeGW}
Let $w\geq 1$. We denote $\Delta^w(a,b)$ the degree of tropical curves whose multiset of ends is
$$
\bigg\{a\cdot\binom{n}{1}, (b-w)\cdot\binom{1}{0}, 1\cdot\binom{w}{0},a\cdot\binom{0}{-1},(an+b)\cdot\binom{-1}{0}\bigg\},
$$
%where exactly one of the ends in the $\binom{1}{0}$ has weight $w$. 
Similarly, for $w',w''\geq 1$ let $\Delta^{w',w''}(a,b)$ be the degree of curves whose multiset of ends is
$$
\bigg\{a\cdot\binom{n}{1}, (b-w'-w'')\cdot\binom{1}{0}, 1\cdot\binom{w'}{0},1\cdot\binom{w''}{0},   a\cdot\binom{0}{-1},(an+b)\cdot\binom{-1}{0}\bigg\},
$$
We denote $N^w(a,b)$ (resp. $N^{w',w''}(a,b)$) the number of rational curves of type $\Delta^w(a,b)$ (resp. $\Delta^{w',w''}(a,b)$) passing through $2b+(n+2)a-w$ (resp. $2b+(n+2)a-w'-w''+1$) points in general position.
\end{definition}

We now set up some notation that will be useful in the sequel. Let $\Gamma$ be a curve corresponding to type \ref{string}  of Lemma \ref{largeJ}. Then $\Gamma$ has a string, and the curve obtained by  removing it, is a union of rational curves. One of them, denoted $\Gamma^0$,  is connected to the string via two horizontal edges which are part of the cycle of $\Gamma$. Denote these edges $h_0'$ and $h_0''$. Denote $\Gamma^1,\ldots, \Gamma^k$ the other connected components, and $h_1,\ldots,h_k$ respectively the horizontal edges connecting them to the string. Let $w_0',w_0'', w_1,\ldots,w_k$ be the weights of the horizontal edges. For each $i$, let $e_i$ be the edge of the string whose upper vertex meets $h_i$, and let $v_i$ be that vertex. Similarly, let $e_0',e_0''$ be the edges meeting $h_0',h_0''$ at vertices $v_0',v_0''$.  See Figure \ref{fig: stringCycleWithNotation} for the case $k=1$.

\begin{figure}[h!] % a curve with a string.
{\centering
\begin{tikzpicture}[font=\tiny,scale=.52]%[scale=.3, font=\tiny]

\draw (-2,-1) -- (-2,0) -- (-1,1) -- (3,3) -- (6,4);
\node[right] at (-2,-.5) {$e_0'$};
\node[right] at (-1.5,.5) {$e_0''$};
\node[right] at (1,2) {$e_1$};

%First rational curve.
\draw (3,3) -- (-3.5,3);
\node[above] at (0,3) {$h_1$};

\draw (-3.5,3) -- (-3.5,4);
\draw (-3.5,3) -- (-4,2.5);
\draw (-4,2.5) -- (-5.5,2.5);
\draw (-4,2.5) -- (-4,2);

\node at (-5.5,3.2) {$\Gamma^1$};

%Rational curve connected to cycle. 
\draw (-2,0) -- (-8,0);
\node[above] at (-5,0) {$h_0'$};

\draw (-1,1) -- (-8,1);
\node[above] at (-4.5,1) {$h_0''$};

\draw (-8,0) -- (-8.5,0.5);
\draw (-8,1) -- (-8.5,0.5);
\draw (-8.5,0.5) -- (-9.5,0.5);
\draw (-9.5,0.5) -- (-10,1);
\draw (-9.5,0.5) -- (-10,0);

\node at (-11,.5) {$\Gamma^0$};
 
\end{tikzpicture}   
\caption{A curve with a cycle connected to a string.}\label{fig: stringCycleWithNotation}
}
\end{figure}
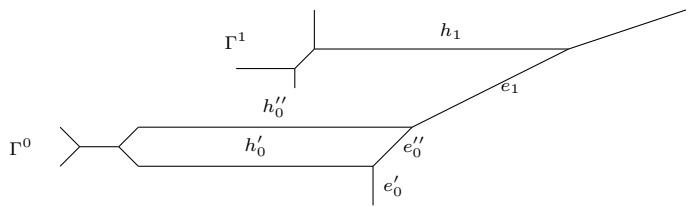

\begin{proposition} \label{reducible}
With  notation as above, 
$$
\text{Mult}(\Gamma) = 2w_0'\cdot w_0''\cdot\text{Mult}(\Gamma^0)\cdot{\prod_{i=1}^k}w_i\cdot\text{Mult}(\Gamma^i).
$$
\end{proposition}
\begin{proof}
We use Lemma \ref{multDet} to compute the multiplicity and construct a matrix representing the map $J$. Fix a point in $\Gamma^0$ and consider the matrix whose columns correspond to the edges of $\Gamma$, and rows correspond to the marked points, the maps $a_1,a_2$ and the $j$-invariant. 

For each marked point, choose a path leading to it from the root vertex, and the entries in each column of the corresponding row are  the length of the edges traversed along this path. The row corresponding to the $j$-invariant consists of the lengths of the edges along the cycle, and the rows representing the maps $a_1,a_2$ consist of the entries of the vectors along the cycle.

For a marked point in $\Gamma^0$ we may choose the path so that it only contains  edges of $\Gamma^0$, and for every marked point on $\Gamma^i$ with $i>0$ we may choose a path that is supported on $\Gamma^0,\Gamma^i, h_0', h_i$ and the part of the string between $v_0'$ and $v_i$. Therefore, the column corresponding to the length of $h_i$ is non-zero only in rows corresponding to a marked point on $\Gamma^i$. In this case, the two entries are $\binom{-w_i}{0}$. Similarly, the cycle consists only of edges of $\Gamma^0$, $h_0',h_0''$, and the part of the string between $v_0'$ and $v_0''$. We see that the column corresponding to $h_0''$ has entry $w_0''$  in the row corresponding to the $j$-invariant, and a $\binom{-w_0''}{0}$ in the  rows corresponding to the cycle. 

We add each column of $h_i$ multiplied by $\frac{w_0'}{w_i}$ and the column of  $h_0''$  multiplied by $\frac{w_0'}{w_0''}$ to the $h_0'$ column, making it zero everywhere, except  for a $2w_0'$ on the $j$-invariant row. The determinant is $2w_0'$ multiplied by the determinant of the matrix obtained by removing the $h_0'$ column and the $j$-invariant row.

Notice that the rows corresponding to a point in $\Gamma^i$ for some $i$ are non-zero only in columns corresponding to edges in $\Gamma^0, h_i$ and edges in $\Gamma^i$. Similarly, the rows corresponding to the cycle is non-zero only at columns corresponding to edges in $\Gamma^0$ and $h_0''$. As a result, we may rearrange the order of the rows and columns to obtain a  lower diagonal block matrix having: (1) block $B_i$ for each $i$ whose rows correspond to the points in $\Gamma^i$, and columns correspond to the edges in $\Gamma^i$, and (2) a block $B''$ whose columns correspond to $e_0'',h_0''$, and the rows correspond to the cycle. The determinant of $B_i$ is $w_i$ times the multiplicity of the curve $\Gamma^i$, with a root vertex at the point where $h_i$ meets the rest $\Gamma^i$. The determinant of $B''$ is $w_0''$. The determinant of the full matrix is the product of the determinants of $B_1,\ldots, B_k$ and $B''$, which gives the desired result. 
\end{proof}

We next deal with  curves $\Gamma$  as in part \ref{stringFlat}. By  removing the string, we are again left with a union of rational curves. One of them, denoted $\Gamma^0$,  is connected to the string via a horizontal edges with a flat cycle. Denote the weights of the edges of the cycle by $w_0',w_0''$. Let $\Gamma^1,\ldots, \Gamma^k$ be the other connected components, and $w_0,\ldots, w_k$ be the weights of the horizontal edges connecting them to the string. 

%For each $i$, let $e_i$ be the edge of the string whose upper vertex meets $h_i$, and let $v_i$ be that vertex. Similarly, let $e_0',e_0''$ be the edges meeting $h_0',h_0''$ at vertices $v_0',v_0''$.  See Figure \ref{stringCycleWithNotation} for the case $k=1$.
%By the arguments used to prove Proposition \ref{reducible} one can show the following.

\begin{proposition} \label{reducibleFlat}
If $w_0'\neq w_0''$ then
$$
\text{Mult}(\Gamma) = 2\cdot{\prod_{i=0}^k}w_i\cdot\text{Mult}(\Gamma^i).
$$
If $w_0'=w_0''$ then 
$$
\text{Mult}(\Gamma) = {\prod_{i=0}^k}w_i\cdot\text{Mult}(\Gamma^i).
$$
\end{proposition}
\begin{proof}
Denote 
$$
\epsilon = \left\{
     \begin{array}{lr}
       1 &  w'\neq w'' \\
       \frac{1}{2} &  w'= w''
     \end{array}
   \right.
   .
$$
By Lemma \ref{multFlat}, the multiplicity of $\Gamma$ equals $2\epsilon$ times the multiplicity of the rational curve $\Gamma'$ obtained by flattening the cycle. Let $A'$ be the matrix whose determinant equals the multiplicity of $\Gamma$. Then arguing as in the proof of Proposition \ref{reducible}, we  rearrange the columns of $A'$ to obtain a block matrix such that the determinant of every block $B_i$ equals $w_i$ times the multiplicity of $\Gamma_i$, and the formula follows.
\end{proof}

Fixing the notation as above, we come to our main result of the section. 

\begin{theorem}\label{mainTheorem}
The number of elliptic curves of degree (a,b) through $2b+(n+2)a-1$ points in general position equals
\begin{align*}
N(a,b) = &\overset{\circ}{\#}\Delta\cdot N^0(a,b) +\\
2\cdot&\sum \binom{N}{N_0,N_1,\ldots, N_k}\left(   w_0'\cdot w_0'' N^{w_0',w_0''}(a_0,b_0)\right)\left(\underset{j=1}{\overset{k}{\prod}}w_j N^{w_j}(a_j,b_j) \right)+\\
&\sum \binom{N}{N_0,N_1,\ldots, N_k}   \underset{j=0}{\overset{k}{\prod}}w_j N^{w_j}(a_j,b_j).
\end{align*}
The sum in the second row is over all partitions $w_0'+w_0''+w_1+\ldots + w_k = n$ with $w_1\geq\ldots\geq w_k$ and $w_0'\geq w_0''$, and over all choices of $a_j$ summing to $a-1$, choices of $b_j$ summing to $b+n$, (where $j=0,\ldots, k$), and $N_0 = 2b_0 + (n+2)a_0-w_0'-w_0''+1$,  $N_i = 2b_i + (n+2)a_i - w_i$. 

\noindent The  sum in the last row is over all partitions $w_0+w_1+\ldots + w_k = n$ with $w_1\geq\ldots\geq w_k$, and partitions $w_0=w_0'+w_0''$, choices of $a_j,b_j$ summing to $a-1, \sum b_j=b+n$ respectively, and  $N_i = 2b_i + (n+2)a_i - w_i$.

\end{theorem}

\begin{proof}
Observe that $\sum N_i = N$.
Since the number of elliptic curves passing through the points is independent of the $j$-invariant, we may assume that it is large in the sense of Lemma \ref{largeJ}, and so we only to count curves $\Gamma$  belonging to one of the three types in the lemma. 

In the first case, we may assume by \cite[Proposition 4.11]{Mikhalkin} (possibly after perturbing the points) that the curve $\Gamma'$ obtained by removing the contracted edge is simple. Now,   Proposition \ref{prop:contracted} implies that the number of curves satisfying the point condition equals the first summand. 

In the second case, by removing the string we obtain a collection of rational curves, with possibly weighted ends, $\Gamma^0,\ldots,\Gamma^k$ passing through the chosen points. Let  the degrees of these curves be $\Delta^{w_0',w_0''}(a_0,b_0), \Delta^{w_1}(a_1,b_1),\ldots, \Delta^{w_k}(a_k,b_k)$. Then  $\sum a_j=a-1,\sum b_j=b+n$ and $w_0'+w_0''+\ldots + w_k = n$. On the other hand, starting with any collection of such rational curves, we obtain an elliptic curve of type $\Delta(a,b)$ by attaching a string to the horizontal weighted ends. By counting the number of ways to do that and applying Proposition~\ref{reducible}  we obtain the second summand. The third case follows similarly, using Proposition~\ref{reducibleFlat}. Note that in this case, we did not require that $w_0'\geq w_0''$, and as a result, when the two weights are different, each curve is being double counted. We compensate for that by counting each such curve with a factor of $\frac{1}{2}$.  
\end{proof}

We finish by examining two special cases of our formula. By fixing $a=d$ and $b=0$ we recover Pandharipande's enumerative formula for $\mathbb{P}^2$:
\begin{corollary} The number of plane elliptic curves of degree $d$ and fixed $j$-invariant passing through $3d-1$ points in general position in the plane is 
$$
\binom{d-1}{2}\cdot N_0(d),
$$
where $N_0(d)$ is the number of rational curves of degree $d$ through the same number of points.
\end{corollary} 

\noindent As was stated in the introduction, Biswas, Mukherjee, and Thakre use symplectic techniques to  compute enumerative invariants for del-Pezzo surfaces~\cite[Theorem 1.2]{BMT} . When $n=1$, the Hirzebruch surface $\mathbb{F}_1$ is del-Pezzo. In that case, it is straightforward to verify that our results agree, and yield the following formula. 

\begin{corollary}\label{cor: f1-formula}
The number of plane elliptic curves of degree $(a,b)$ and fixed $j$-invariant passing through $2b+3a-1$ points in general position on $\mathbb{F}_1$ is 
$$ 
(\frac{a^2+2ab-3a-2b+2}{2})\cdot N_0(a,b),
$$
where $N_0(a,b)$ is the number of rational curves of degree $(a,b)$ through the same number of points.
\end{corollary}
\begin{proof}
Since $n=1$, there is no choice of positive weights as in the second and third summands of Theorem \ref{mainTheorem}. Therefore, the number of curves in question equals the first summand. On $\mathbb{F}_1$, the number of internal lattice points in the Newton polygon of  a curve of degree $(a,b)$ is $\frac{a(a-1)}{2}$, which equals $\frac{a^2+2ab-3a-2b+2}{2}$, and the result follows.
\end{proof}

\bibliographystyle{siam}
\bibliography{bibHirz}

\end{document}